\documentclass{cpamart1}     
\received{Month 200X}       
\volume{000}
\startingpage{1}                      

\authorheadline{A. Iserles and M. Webb}
\titleheadline{Fast Computation with a Skew-symmetric Differentiation Matrix}

\usepackage{amssymb}
\usepackage{graphicx}
\usepackage{mathdots}
\allowdisplaybreaks

\usepackage{mathptmx}

\newtheorem{theorem}{Theorem}[section]
\newtheorem{lemma}[theorem]{Lemma}

\theoremstyle{definition}

\theoremstyle{remark}

\newcommand{\BB}[1]{\mathbb{#1}}
\newcommand{\CC}[1]{\mathrm{#1}}
\newcommand{\ee}{\mathrm{e}}
\newcommand{\ii}{\mathrm{i}}
\renewcommand{\O}[1]{\mathcal{O}\left( #1 \right)}

\newcommand{\R}[1]{\eqref{#1}}
\newcommand{\D}{\mathrm{d}}

\renewcommand{\ee}{\mathrm{e}}
\renewcommand{\ii}{\mathrm{i}}
\newcommand{\MM}{\mu}

\newcommand{\EE}{{\CC{E}}}
\newcommand{\OO}{{\CC{O}}}
\newcommand{\PP}{\CC{P}}
\newcommand{\GG}[1]{\mathfrak{#1}}

\newcommand{\arctanh}{\mathrm{arctanh}\,}

\begin{document}                        


\title{Fast Computation of Orthogonal Systems with a Skew-symmetric Differentiation Matrix}

\author{Arieh Iserles}{University of Cambridge}
\author{Marcus Webb}{University of Manchester}






\begin{abstract}
  Orthogonal systems in $\CC{L}_2(\BB{R})$, once implemented in spectral methods, enjoy a number of important advantages if their differentiation matrix is skew-symmetric and highly structured. Such systems, where the differentiation matrix is skew-symmetric, tridiagonal and irreducible, have been recently fully characterised. In this paper we go a step further, imposing the extra requirement of fast computation: specifically, that the first $N$ coefficients {of the expansion} can be computed to high accuracy in $\O{N\log_2N}$ operations.  We consider two settings, one approximating a function $f$ directly in $(-\infty,\infty)$ and the other approximating $[f(x)+f(-x)]/2$ and $[f(x)-f(-x)]/2$ separately in $[0,\infty)$. In each setting we prove that there is a single family, parametrised by $\alpha,\beta > -1$, of orthogonal systems with a skew-symmetric, tridiagonal, irreducible differentiation matrix and whose coefficients can be computed as Jacobi polynomial coefficients of a modified function. The four special cases where $\alpha, \beta= \pm 1/2$ are of particular interest, since coefficients can be computed using fast sine and cosine transforms. Banded, Toeplitz-plus-Hankel multiplication operators are also possible for representing variable coefficients in a spectral method. In Fourier space these orthogonal systems are related to an apparently new generalisation of the Carlitz polynomials.
\end{abstract}

\maketitle   






\section{Introduction}

This paper continues a project which we have commenced in \cite{iserles18osssd}, to investigate complete systems in $\CC{L}_2(\BB{R})$ with a skew-symmetric differentiation matrix. Let us begin by explaining briefly the  underlying concepts, motivating our work and revisiting in a very abbreviated manner the contents of \cite{iserles18osssd}.

The concept of a {\em differentiation matrix\/} originates in the theory of finite differences --- essentially, this is the matrix taking a vector of function values to those of an approximation of the first derivative. An elementary example,
\begin{equation}
\label{eq:1.1}
\frac{1}{2\Delta x}\left[
\begin{array}{ccccc}
0 & 1 & 0 & \cdots & 0\\
-1 & 0 & 1 & \ddots & \vdots\\
0 & \ddots & \ddots & \ddots & 0\\
\vdots & \ddots & \ddots & \ddots & 1\\
0 & \cdots & 0 & -1 & 0
\end{array}
\right]\!,
\end{equation}
is obtained from a second-order central difference approximation on a uniform grid with spacing $\Delta x>0$ --- note that it is skew-symmetric. Skew-symmetry mimics the self-adjointness of the first derivative operator in the standard $\CC{L}_2$ Hilbert space with either zero Dirichlet or periodic boundary conditions, and it confers a wide range of advantages on a numerical method. We refer to \cite{hairer16nsp,iserles14ssd,iserles16jps} for a wealth of specific examples: in essence, once a differentiation matrix is skew symmetric, it is often easy to prove stability for linear PDEs, as well as conservation of energy whenever it is mandated by the underlying equation. 

The matrix in \R{eq:1.1} is skew symmetric, yet it is a sobering thought that this second-order approximation of the derivative is as good as it gets: no skew-symmetric finite-difference differentiation matrix on a uniform grid may exceed order 2 \cite{iserles14ssd}. It is possible (but not easy) to obtain higher-order differentiation matrices of this kind carefully choosing specific non-uniform grids \cite{hairer16nsp,hairer17bss}, but this is far from easy for high orders and, at any rate, finite differences are not the approach of choice in this paper. 

An example that motivates much of our work is the {\em linear Schr\"odinger equation\/} in a semiclassical regime,
\begin{equation}
\label{eq:1.2}
\ii\varepsilon \frac{\partial u}{\partial t}=\varepsilon^2\frac{\partial^2 u}{\partial x^2}-V(x,t)u,\qquad a\leq x\leq b,\quad t\geq0,
\end{equation}
(for simplicity, we focus here on the univariate case), typically given with periodic boundary conditions. The parameter $\varepsilon>0$ is  small and $V$ is the potential energy \cite{jin11mcm}. Typically, \R{eq:1.2} is solved by a spectral method (or spectral collocation) in space, followed by a combination of splittings (e.g.\ Strang splitting or Zassenhaus splitting) and, in the case of time-dependent potential, the Magnus expansion or a version thereof in time \cite{bader16eml,blanes17hoc,iserles18mlm}. 

The definition of a differentiation matrix extends naturally into the realm of spectral methods. In any spectral method we expand the unknown in a (truncated) orthonormal basis $\{\varphi_m\}_{m\in\BB{Z}_+}$ of the underlying $\CC{L}_2$ Hilbert space, where the $\varphi_m$s are suitably regular. The (infinite-dimensional) linear map $\mathcal{D}$ taking $\{\varphi_m\}_{m\in\BB{Z}_+}$ into $\{\varphi_m'\}_{m\in\BB{Z}_+}$ is called a {\em differentiation matrix.\/} In other words, $\varphi_m'=\sum_{k=0}^\infty {\mathcal D}_{m,k}\varphi_k$, $m\in\BB{Z}_+$. The virtues of skew-symmetry remain undimmed in this setting and they immediately imply stability and energy conservation for \R{eq:1.2}. 

Because of periodic boundary conditions, it is natural to use a Fourier basis for \R{eq:1.2}, leading to a skew-Hermitian (and diagonal) differentiation matrix.\footnote{We typically index a Fourier basis by integers, rather than non-negative integers, but this causes no difficulty in our framework.} To all intents and purposes, a skew-Hermitian matrix shares the benefits of a skew-symmetric one. Moreover, a Fourier basis has another critical virtue: we can approximate the first $N$ expansion coefficients using FFT in just $\O{N\log_2N}$ operations. Therefore, there is little incentive to seek alternative bases. 

However, the traditional setting -- finite interval, periodic boundary conditions -- of \R{eq:1.2} and other dispersive equations of quantum mechanics is being increasingly challenged, since it is inappropriate for long-term integration (or even short-term integration in the presence of quantum scattering) because sooner or later, wave packets reach a boundary and periodicity leads to non-physical solutions. Modern applications of quantum mechanics, not least quantum control \cite{shapiro03qc}, are an important example of this behaviour. Arguably, the most natural setting for \R{eq:1.2} is that of a Cauchy problem on the entire real line, with solutions confined to $\CC{L}_2(\BB{R})$. Zero Dirichlet boundary conditions on an interval correspond to a \emph{closed\/} quantum system which can be interpreted as a quantum system on the real line with $u(x,0) = 0$ for $x \notin (a,b)$ or where the potential $V$ is sufficiently large at the endpoints that no tunnelling is possible \cite{patriarca1994boundary}.

In \cite{iserles18osssd} we have developed a comprehensive theory of orthonormal bases of $\CC{L}_2(\BB{R})$ with a skew-symmetric, tridiagonal, irreducible  differentiation matrix
\begin{equation}
\label{eq:1.3}
\mathcal{D}=
\left[
\begin{array}{cccccc}
0 & b_0 & 0\\
-b_0 & 0 & b_1 & 0\\
0 & -b_1 & 0 & b_2 & 0\\
& \ddots & \ddots & \ddots & \ddots
\end{array}
\right]\!,
\end{equation}
where without loss of generality $b_m>0$, $m\in\BB{Z}$. There exists a one-to-one relationship between absolutely continuous real Borel measures $\D\mu$ with all their moments bounded, symmetric with respect to the origin and supported by the entire real line on the one hand, orthonormal bases $\Phi=\{\varphi_m\}_{m\in\BB{Z}_+}$, complete in $\CC{L}_2(\BB{R})$ and with a differentiation matrix of the form \R{eq:1.3} on the other. The emphasis on a {\em tridiagonal\/} differentiation matrix is motivated by our wish to minimise the cost of eventual computations. It is standard to solve equations like \R{eq:1.2} by splittings \cite{bader14eas,blanes17hoc}. Thus, eventually we need to compute terms of the form $\ee^{h\mathcal{D}}\MM{u}$ or $\ee^{\ii h\mathcal{D}^2}\MM{u}$, and this becomes considerably cheaper once $\mathcal{D}$ is tridiagonal -- a forthcoming paper of Celledoni and Iserles, for example, will describe an $\O{N\log_2N}$ algorithm to this end.

Specifically, let $\D\mu(\xi)=w(\xi)\D \xi$ be such Borel measure and $\{p_m\}_{m\in\BB{Z}_+}$ the underlying monic orthogonal polynomials. Because of symmetry, they obey a three-term recurrence relation of the form 
\begin{equation}
\label{eq:1.4}
p_{m+1}(x)=xp_m(x)-\lambda_m p_{m-1}(x),\qquad m\in\BB{Z}_+,
\end{equation}
where $\lambda_0=0$ and $\lambda_m>0$, $m\in\BB{N}$. Set $b_m=\sqrt{\lambda_{m+1}}$, $m\in\BB{Z}_+$, and
\begin{equation}
\label{eq:1.5}
\varphi_m(x)=\frac{(-\ii)^n}{\sqrt{2\pi}} \frac{1}{\|p_m\|} \int_{-\infty}^\infty g(\xi) p_m(\xi) \ee^{-\ii x\xi}\D\xi,\qquad m\in\BB{Z}_+,
\end{equation}
where $\|p_m\|^2=\int_{-\infty}^\infty p_m^2(x)\D\mu(x)$ and $g$ is a complex-valued function with even real part and odd imaginary part such that $|g(\xi)|^2 = w(\xi)$. It has been proved in \cite{iserles18osssd} that $\{\varphi_m\}_{m\in\BB{Z}_+}$ is indeed orthonormal and complete in $\CC{L}_2(\BB{R})$ (essentially due to the Plancharel theorem) and has the differentiation matrix \R{eq:1.3}. On the other hand, given such a basis, using the Favard Theorem we can always associate it with a symmetric Borel measure supported by the real line: such measure is unique if the corresponding Hamburger moment problem is determinate \cite[p.~73]{chihara78iop}.

Realistic implementation of the aforementioned bases, however, requires that two computations can be performed explicitly, namely explicit formation of the {\em Orthogonal Polynomial System (OPS)\/} $\{p_m\}_{m\in\BB{Z}_+}$ given $\D\mu$ and the integrals \R{eq:1.5}, since otherwise we will not have the basis $\Phi$ in an explicit form. Here we are rapidly running against the limitations of current theory of orthogonal polynomials. An obvious and familiar example of {\em Hermite functions\/}
\begin{displaymath}
\varphi_m(x)=\frac{(-1)^m}{(2^mm!)^{1/2}\pi^{1/4}} \ee^{-x^2/2}\CC{H}_m(x),\quad b_m=\left(\frac{m+1}{2}\right)^{\!1/2},\qquad m\in\BB{Z}_+,
\end{displaymath}
where the $\CC{H}_m$ are Hermite polynomials, aside, very few symmetric measures supported by the entire real line are known for which the underlying OPS (or even the recurrence coefficients $\lambda_m$) can be written explicitly:  \cite{iserles18osssd} list just the generalised Hermite polynomials and Carlitz polynomials. Indeed, the resolution of the {\em Freud problem,\/} precise asymptotic determination of the $\lambda_m$s for $w(\xi)=\ee^{-\xi^{n}+q(\xi)}$, where $n\geq2$ is an even integer and $q$ an even polynomial of degree $\leq n-2$, using the Riemann--Hilbert transform \cite{deift99sao,fokas92iam}, is considered  a recent triumph of the theory of orthogonal polynomials -- and in our setting we need explicit values, not just asymptotics! Thus, our basic challenge in this paper is to search for systems $\Phi$ by alternative means.

We seek systems that confer a tangible advantage when compared to the default of Hermite functions, and in this paper we find systems that allow for faster computation of expansion coefficients via Fast Cosine Transforms (FCT) or Fast Sine Transforms (FST). Given $f\in\CC{L}_2(\BB{R})\cap\CC{C}^\infty(\BB{R})$, we wish to compute 
\begin{displaymath}
\hat{f}_m=\int_{-\infty}^\infty f(x)\varphi_m(x)\D x,\qquad m=0,\ldots,N-1,
\end{displaymath}
to sufficiently high accuracy in the least possible number of operations. We can identify three options. 

Firstly, it is possible to use the {\em fast multipole method\/} to compute expansion coefficients in $\O{N\log N+N\log\varepsilon^{-1}}$ operations to accuracy $\varepsilon$, provided that a table of nodes and weights of a suitable Gauss--Christoffel quadrature is available \cite{dutt96fap}. In the Hermite case, Gauss--Hermite quadrature is required and we note that its nodes and weights can be computed in $\O{N}$ operations \cite{townsend16fag}. For other bases of the form \eqref{eq:1.5}  quadrature rules may take up to $\O{N^2}$ operations to compute.

Secondly, by the Fourier-transform on $\CC{L}_2(\BB{R})$, the evaluation of the $\hat{f}_m$s reduces to an expansion in the polynomial basis $\{p_m\}_{m\in\BB{Z}_+}$. This, however, confers an advantage only if we can compute the polynomial coefficients in a fast manner. Of all fast and stable methodologies known to the present authors, this restricts us to $\D\mu$ with bounded support, which has been ruled out earlier. The reason why we require $\CC{supp}\,\mu=\BB{R}$ is the completeness of the expansion: if, without loss of generality, $\CC{supp}\,\mu=(-1,1)$ then $\Phi$ is complete not in $\CC{L}_2(\BB{R})$ but in the {\em Paley--Wiener space\/} $\mathcal{PW}_{[-1,1]}(\BB{R})$ \cite{iserles18osssd}. This is of no clear interest in the design of spectral methods for PDEs but might be relevant, for example, to signal processing. 

In this paper we follow a third option. We seek a basis of $\varphi_m$s so that the expansion coefficients $\hat{f}_m$ are equal to standard orthogonal polynomial coefficients of a modified function. If the standard polynomials are, for example, the Chebyshev polynomials (which \emph{a priori} is not necessarily even possible) then the first $N$ coefficients can be approximated in $\O{N\log_2N}$ operations using the FCT. In Section \ref{sec:fullrange}, this line of reasoning leads us to consider the model,
\begin{equation}
\label{eq:1.6}
\varphi_m(x)=\Theta(x) q_m(H(x)),\qquad m\in\BB{Z}_+,
\end{equation}
where $\Theta \in \CC{L}_2(\BB{R})$, $\{q_m\}_{m\in\BB{Z}_+}$ is an orthonormal OPS with respect to a measure $W(t)\D t$ supported in $(-1,1)$, and $H$ is a differentiable, strictly monotonically increasing function mapping $(-\infty,\infty)$ onto $(-1,1)$. 

We prove in Section 2 that enforcing orthonormality on such a model means that the coefficients are equal to the coefficients of $f(h(t))/\Theta(h(t))$ in the $\{q_m\}_{m\in\BB{Z}_+}$ basis, where $h$ is the inverse function of $H$. When we further require that our system has the differentiation matrix in \eqref{eq:1.3}, all options, up to an affine change of variables, are whittled down to
\begin{equation}
\label{eq:1.7}
\varphi^{(\alpha,\beta)}_m(x)=\frac{(-1)^m}{\sqrt{g_m^{(\alpha,\beta)}}} (1-X)^{\frac{\alpha+1}{2}}(1+X)^{\frac{\beta+1}{2}} \PP_m^{(\alpha,\beta)}(X), \qquad X = \tanh x,
\end{equation}
where $\alpha,\beta>-1$, $\PP_m^{(\alpha,\beta)}$ is the $m$th {\em Jacobi polynomial\/} and
\begin{equation}
\label{eq:1.8}
g_m^{(\alpha,\beta)}=\int_{-1}^1 (1-t)^\alpha(1+t)^\beta  [\PP_m^{(\alpha,\beta)}(t)]^2\D t=\frac{2^{1+\alpha+\beta}\CC{\Gamma}(1+\alpha+m)\CC{\Gamma}(1+\beta+m)}{m!(1+\alpha+\beta+2m)\CC{\Gamma}(1+\alpha+\beta+m)}
\end{equation}
\cite[p.~260]{rainville60sf}. 

The first $N$ expansion coefficients of a function with respect to $\Phi$ as in equation \eqref{eq:1.7} can always be approximated in $\O{N(\log N)^2}$ operations by fast polynomial transform techniques \cite{townsend18fpt}. The cases $\alpha,\beta = \pm 1$ correspond to the Chebyshev polynomials of four different kinds (see \cite[18.3]{DLMF}), all of whose coefficients can be approximated in $\O{N\log_2 N}$ operations via various flavours of FCT and FST. Furthermore, a given expansion $f(x) = \sum_{k=0}^{N-1} c_k \varphi^{(\alpha,\beta)}_k(x)$ can be evaluated at a single point $x \in \BB{R}$ using Clenshaw's algorithm \cite{clenshaw1955note}. 

In Section~3 we turn our gaze to a more complicated model,
\begin{eqnarray}
  \label{eq:1.10}
  \varphi_{2m}(x)&=&\Theta_\EE(x) r_m(H(x)),\\
  \nonumber
  \varphi_{2m+1}(x)&=&\Theta_\OO(x) s_m(H(x)), \qquad m\in\BB{Z}_+.
\end{eqnarray}
Here $\Theta_\EE$ and $\Theta_\OO$ are given functions, the first even and the second odd, $\{r_m\}_{m\in\BB{Z}_+}$ and $\{s_m\}_{m\in\BB{Z}_+}$ are  orthonormal OPSs with respect to the measures $w_\EE\!\D t $ and $w_\OO\!\D t$, both supported by $(-1,1)$, and $H$ maps strictly monotonically $(0,\infty)$ to $(-1,1)$. Neither $w_\EE$ nor $w_\OO$ need {to} be symmetric with respect to the origin. Advancing along similar lines to our analysis of \R{eq:1.6}, computation of coefficients corresponding to the system \R{eq:1.7} can be reduced to the expansion in the two OPS, $\{r_m\}_{m\in\BB{Z}_+}$ and $\{s_m\}_{m\in\BB{Z}_+}$. 

The exploration of \R{eq:1.10} is a lengthier process. Step after step we winnow the possibilities for $H$, $\Theta_\EE$, $\Theta_\OO$, $w_\EE$ and $w_\OO$ -- at the end we are left with just a one-parameter family of such systems. We prove that \R{eq:1.10} is consistent with orthonormality and with a skew-symmetric, tridiagonal, irreducible differentiation matrix if and only if $w_\EE(t)=(1-t)^\alpha(1+1)^{-\frac12}$ and $w_\OO(t)=(1-t)^\alpha(1+t)^{\frac12}$ for some $\alpha>-1$ -- in other words, the $r_m$s and $s_m$s are (normalised) {\em Jacobi polynomials\/} $\PP_m^{(\alpha,-\frac12)}$ and $\PP_m^{(\alpha,\frac12)}$ respectively. However, looking deeper, we demonstrate that these systems are {\em identical\/} to the full-range system \R{eq:1.7} with $\alpha=\beta$, implemented in a half-range mode.

Why do we need to discuss both the `full-length' systems of Section~2 and the `two half-length' systems of Section~3, in particular as the latter are nothing but an alternative implementation of the former? As things stand, they exhibit the same range of benefits: they are orthonormal and complete in $\CC{L}_2(\BB{R})$, have skew-symmetric, tridiagonal differentiation matrices, while their expansion coefficients can be approximated in a fast and stable manner. We touch upon the advantages of either approach in Section~5. However, these are early days in the investigation of orthogonal systems with skew-symmetric  differentiation matrices and their applications to spectral methods. There is great merit, we believe, in exploring their theory, potential and limitations in the broadest possible sense.

In Section~\ref{sec:tanh-Jacobifunctions} we give the functions $g$ and measures $\D\mu(\xi)$ which appear in the formula \eqref{eq:1.5} for these orthonormal bases, for all $\alpha,\beta > -1$. They turn out to be the weights for a generalised version of the Carlitz polynomials (now with two parameters) discussed in \cite{iserles18osssd}, which have the explicit form,
\begin{displaymath}
\D\mu_{\alpha,\beta}(\xi) = C_{\alpha,\beta}^2 \left|\Gamma\left(\frac{\alpha+1}{2} + \frac{\ii\xi}{2} \right)\Gamma\left(\frac{\beta+1}{2} - \frac{\ii\xi}{2} \right)\right|^2 \!\D\xi,\qquad \alpha,\beta>-1.
\end{displaymath}
The case $\beta = -\alpha$ corresponds to the one-parameter family of Carlitz polynomials. The original interest in Carlitz polynomials stems from the fact that their moments can be written in terms of Bernoulli polynomials. We are not aware of work on this two-parameter generalisation, nor can we anticipate its relevance to the theory of orthogonal polynomials on the real line.

In Section 5 we list some conclusions of our work, and flesh out the basis for applications of these functions to computing the Fourier transform and to a fast Olver--Townsend-type spectral method on the real line \cite{olver2013fast}.

This is the place to mention \cite{iserles19for}, a companion paper to the current one, where we explore complex-valued orthonormal systems with skew-Hermitian tridiagonal differentiation matrices. {\em Inter alia\/} we demonstrate there the existence of another system with coefficients that can be computed in $\O{N\log_2N}$ operations, this time using the {\em Fast Fourier Transform.\/} 

{An alternative to the approach of this paper is to map a differential equation from $\CC{L}_2(\BB{R})$, say, to a finite domain and solve the new equation there with a spectral method. This  is often accomplished with conformal maps, providing a convenient means to study the speed of convergence \cite{boyd87orf,guo10gjr,narayan13gwr,yi12gjr}.  This  very natural and useful approach is genuinely different from the theme of this paper: the procedures of approximation and `shrinkage' in general do not commute. As an example, consider a Cauchy problem for \R{eq:1.2}. A major feature of the linear Schr\"odinger equation is that the Euclidean norm of the solution remains constant as the time evolves, reflecting the quantum-mechanical interpretation of the solution: $|u(\,\cdot,t)|^2$ is a probability density. Mapping the equation into a finite interval, the $\CC{L}_2$ norm  corresponds to a fairly nonstandard Sobolev norm. While good enough for the proof of convergence by means of the Lax Equivalence Theorem, this makes matters quite difficult insofar as energy conservation is concerned.} 

\setcounter{equation}{0}
\setcounter{figure}{0}
\section{The full-range setting}\label{sec:fullrange}

In this section we characterise all orthonormal systems $\Phi = \{\varphi_m\}_{m\in\BB{Z}_+}$ with a skew-symmetric, irreducible, tridiagonal differentiation matrix, with the additional constraint that the expansion coefficients of a function $f \in \CC{L}_2(\BB{R})$ are equal to expansion coefficients for a weighted and mapped version of $f$, in an orthonormal polynomial basis on $(-1,1)$. Explicitly, for the differentiation matrix, we require
\begin{equation}
\label{eq:2.1}
\varphi_m'=-b_{m-1}\varphi_{m-1}+b_m\varphi_{m+1},\qquad m\in\BB{Z}_+,
\end{equation}
where $b_{-1}=0$ and, without loss of generality, $b_m > 0$ for $m\in\BB{Z}_+$. For the expansion coefficients, we require
\begin{equation}
\label{eq:2.2}
\hat{f}_m = \int_{-1}^1 \theta(t)f(h(t)) q_m(t) W(t) \, \D t,
\end{equation}
where $\{q_m\}_{m\in\BB{Z}}$ are orthonormal polynomials with respect to the absolutely continuous measure $W(t)\D t$ on $(-1,1)$, and $h,\theta: (-1,1) \to \BB{R}$.

We will make the following assumptions about $\theta$, $h$ and $W$. Later we \emph{deduce} much more indeed about these functions as necessary consequences of our model.
\begin{itemize}
  \item $h$ maps onto the whole of $\BB{R}$, is differentiable with $h'$ being a measurable, positive function.
  \item This implies the existence of an inverse function $H : \BB{R} \to (-1,1)$ which is differentiable with $h'(t)H'(h(t)) = 1$, therefore $H'$ is also positive and measurable.
  \item $\theta$ is such that $t \mapsto \theta(t)\sqrt{W(t)/h'(t)} \in \CC{L}_\infty(-1,1)$. A careful application of the weighted Cauchy-Schwarz inequality with weight $h'$ shows that this implies that $\hat{f}_m$ is finite for all $m\in\BB{Z}_+$.
\end{itemize}

Changing variables to $x = h(t)$ yields
\begin{equation}
\label{eq:2.3}
\hat{f}_m = \int_{-\infty}^{\infty} f(x)\Theta(x)q_m(H(x)) \, \mathrm{d} x,
\end{equation}
where $\Theta(x) = \theta(H(x))H'(x) W(H(x))$. For this to hold for all $f \in \CC{L}_2(\BB{R})$,  $\varphi_m$ must be of the form,
\begin{equation}
\label{eq:2.4}
\varphi_m(x)=\Theta(x) q_m(H(x)),\qquad m\in\BB{Z}_+.
\end{equation}

The rest of the section is devoted to proving the following surprisingly simple result.

\begin{theorem}\label{thm:alphabeta}
  All orthonormal systems of the form in equation \R{eq:2.4} with a skew-symmetric, tridiagonal, irreducible differentiation matrix are, up to an affine change of variables, of the form
  \begin{equation}
  \label{eq:2.5}
  \varphi_m(x)=\frac{(-1)^m}{\sqrt{g_m^{(\alpha,\beta)}}} (1-\tanh x)^{(\alpha+1)/2}(1+\tanh x)^{(\beta+1)/2} \PP_m^{(\alpha,\beta)}(\tanh x),\qquad m\in\BB{Z}_+
  \end{equation}
  for $\alpha,\beta>-1$, with $g_m^{(\alpha,\beta)}$ as in \R{eq:1.8}. The coefficients $b_m$ are given by \R{eq:2.8}. 
\end{theorem}
We will discuss properties of these orthonormal systems in Section~4. In particular, we will show that they are in fact complete orthonormal bases for $\CC{L}_2(\BB{R})$ by deriving the measure $\D\mu(\xi)$ and function $g(\xi)$ in the Fourier transform expression which these systems must satisfy as in equation \eqref{eq:1.5} and \cite{iserles18osssd}.

\subsection{Necessary conditions}

Let us work out {the} necessary consequences of $\Phi$ being an orthonormal system in $\CC{L}_2(\BB{R})$. For all $n,m\in\BB{Z}_+$,
\begin{eqnarray*}
  \int_{-\infty}^\infty \varphi_n(x)\varphi_m(x) \, \D x &=& \int_{-\infty}^\infty \Theta(x)^2 q_m(H(x))q_n(H(x)) \, \D x \\
  &=& \int_{-1}^1 q_m(t) q_n(t) \Theta(h(t))^2 h'(t) \, \D t.
\end{eqnarray*}
It follows at once that
\begin{equation}
\label{eq:2.6}
W(t) = \Theta(h(t))^2 h'(t) \qquad\text{and}\qquad \Theta(x) = \sqrt{H'(x) W(H(x))}.
\end{equation}
We will return to this later.

It is considerably more complicated to ensure the existence of a skew-symmetric, tridiagonal differentiation matrix. First note that by setting $m = 0$ in equation \eqref{eq:2.4}, we see that $\Theta$ is infinitely differentiable and in $\CC{L}_2(\BB{R})$, because it is proportional to $\varphi_0$ (which is an infinitely differentiable function in $\CC{L}_2(\BB{R})$, being the Fourier transform of a superalgebraically decaying $\CC{L}_2(\BB{R})$ function by \eqref{eq:1.5}).

Inserting equation \eqref{eq:2.4} into equation \eqref{eq:2.1}, we obtain for all $m \in \BB{Z}_+$,
\begin{displaymath}
\Theta' q_m(H) + \Theta q_m'(H)H' = -b_{m-1}\Theta q_{m-1}(H) + b_m \Theta q_{m+1}(H).
\end{displaymath}
Setting $m=0$ implies that $\Theta'(x) = b_0 q_1(H(x)) \Theta(x) / q_0(H(x))$, which can be substituted back into the equation for general $m \in \BB{Z}_+$ to find
\begin{displaymath}
\Theta b_0 q_1(H) q_m(H)/q_0(H) + \Theta q_m'(H)H' = -b_{m-1}\Theta q_{m-1}(H) + b_m \Theta q_{m+1}(H).
\end{displaymath}
Dividing through by $H'(x)\Theta(x)$ and changing variables to $t =H(x)$ gives us
\begin{equation}
\label{eq:2.7}
q_m'(t) = h'(t) \left[ -b_{m-1}q_{m-1}(t) + b_m q_{m+1}(t) - \frac{b_0}{q_0}q_1(t)q_m(t)\right]\!.
\end{equation}
Here we used the fact discussed earlier that $H'(h(t))h'(t) = 1$.

Now, the left-hand-side of equation \eqref{eq:2.7} is a polynomial of degree exactly $m-1$, while the right-hand-side is $h'(t)$ times a polynomial of degree at most $m+1$. It follows from the case $m=1$ that $h'(t)=1/(at^2+bt+c)$ for some real numbers $a,b$ and $c$. Since $h:(-1,1)\rightarrow\BB{R}$ and $h'>0$, necessarily $h'(\pm1)=\infty$ and it follows that, up to an affine change of variables, $h(t)=\arctanh t$ and $H(x)=\tanh x$. 

Back to equation \eqref{eq:2.6}. Substituting in the necessary form of $h$ we have,
\begin{displaymath}
W(t) = \frac{\Theta(\arctanh t)^2}{1-t^2}.
\end{displaymath} 
Using the fact that $\Theta'(x) = b_0 q_1(H(x)) \Theta(x) / q_0(H(x))$, we conclude after simple algebra that
\begin{displaymath}
W'(t) = \frac{2b_0q_1(t)/q_0 + 2t}{1-t^2} W(t).
\end{displaymath}
Writing $2b_0 q_1(t) /q_0 + 2t = (1+t)\alpha + (1-t)\beta$ for real numbers $\alpha$ and $\beta$,  the solution to this ODE with $W(0) = 1$ without loss of generality, is
\begin{displaymath}
W(t)=(1-t)^\alpha(1+t)^\beta.
\end{displaymath}
It is none other than a {\em Jacobi weight\/}. We deduce that, if at all possible, the polynomials in the system \R{eq:2.4} must be the (orthonormal) Jacobi polynomials, 
\begin{displaymath}
q_m(t)=\frac{(-1)^{k_m}}{\sqrt{g_m^{(\alpha,\beta)}}}\PP_m^{(\alpha,\beta)}(t)
\end{displaymath}
where $g_m^{(\alpha,\beta)}$ has been defined in \eqref{eq:1.8} and $k_m$ is an integer --- for the time being, to allow for simpler algebra, we assume that $k_m=0$ but will change this later.

Substituting $W(t)$ into equation \eqref{eq:2.6} shows that
\begin{displaymath}
\Theta(x) = \left( 1 - \tanh x\right)^{\frac{\alpha+1}{2}}\left( 1 + \tanh x\right)^{\frac{\beta+1}{2}}.
\end{displaymath}

\subsection{Sufficient conditions}

All that remains is to check whether these systems actually satisfy the requirements set out at the start of the section. Firstly, the functions are in $\CC{L}_2(\BB{R})$ if and only if $\alpha,\beta > -1$ (which also ensures that $W(t) \D t$ is a finite measure). A quick change of variables will show that these functions are indeed orthonormal for all $\alpha,\beta > -1$. For the final requirement on the differentiation matrix, we have,
\begin{eqnarray*}
  \varphi_m'(t)&=&-\frac{\alpha+1}{2} \frac{1}{\sqrt{g_m^{(\alpha,\beta)}}} (1-t^2) (1-t)^{\frac{\alpha-1}{2}}(1+t)^{\frac{\beta+1}{2}} \PP_m^{(\alpha,\beta)}(t)\\
  &&\mbox{}+\frac{\beta+1}{2}\frac{1}{\sqrt{g_m^{(\alpha,\beta)}}} (1-t)^{\frac{\alpha+1}{2}}(1+t)^{\frac{\beta-1}{2}} \PP_m^{(\alpha,\beta)}(t)\\
  &&\mbox{}+\frac{1}{\sqrt{g_m^{(\alpha,\beta)}}} (1-t^2) (1-t)^{\frac{\alpha+1}{2}}(1+t)^{\frac{\beta+1}{2}} \frac{\D\PP_m^{(\alpha,\beta)}(t)}{\D t}\\
  &=&\frac{(1\!-\!t)^{\frac{\alpha+1}{2}}(1+t)^{\frac{\beta+1}{2}}}{\sqrt{g_m^{(\alpha,\beta)}}} \left[ \!\left(\frac{\beta\!-\!\alpha}{2}-\frac{\alpha\!+\!\beta\!+\!2}{2}t\right) \PP_m^{(\alpha,\beta)}(t) +(1-t^2)\frac{\D\PP_m^{(\alpha,\beta)}}{\D t}\right]\!.
\end{eqnarray*}
According to \cite[18.9.17--18]{DLMF}
\begin{eqnarray*}
  (1-t^2)\frac{\D\PP_m^{(\alpha,\beta)}(t)}{\D t}(t)&=&m\!\left(\frac{\alpha-\beta}{\alpha+\beta+2m}-t\right)\!\PP_m^{(\alpha,\beta)}(t)+\frac{2(\alpha+m)(\beta+m)}{\alpha+\beta+2m}\PP_{m-1}^{(\alpha,\beta)}(t)\\
  &=&(\alpha+\beta+m+1) \left(\frac{\alpha-\beta}{\alpha+\beta+2m+2}+t\right)\PP_{m-1}^{(\alpha,\beta)}(t)\\
  &&\mbox{}=\frac{2(m+1)(\alpha+\beta+m+1)}{\alpha+\beta+2m+2} \PP_{m+1}^{(\alpha,\beta)}(t),
\end{eqnarray*}
while the three-term recurrence relation for Jacobi polynomials is 
\begin{displaymath}
t\PP_m^{(\alpha,\beta)}(t)=A_m\PP_{m-1}^{(\alpha,\beta)}+B_m\PP_m^{(\alpha,\beta)}(t)+C_m\PP_{m+1}^{(\alpha,\beta)}(t),
\end{displaymath}
where
\begin{eqnarray*}
  A_m&=&\frac{2(\alpha+m)(\beta+m)}{(\alpha+\beta+2m)(\alpha+\beta+2m+1)},\\
  B_m&=&-\frac{\alpha^2-\beta^2}{(\alpha+\beta+2m)(\alpha+\beta+2m+2)},\\
  C_m&=&\frac{2(m+1)(\alpha+\beta+m+1)}{(\alpha+\beta+2m+1)(\alpha+\beta+2m+2)}
\end{eqnarray*}
\cite[p.~263]{rainville60sf}. In other words,
\begin{displaymath}
\varphi_m'(x)=\eta_m\varphi_{m-1}(x)+\theta_m\varphi_m(x)+b_m\varphi_{m+1}(x),
\end{displaymath}
where
\begin{eqnarray*}
  \eta_m&=&\sqrt{\frac{g_{m-1}^{(\alpha,\beta)}}{g_m^{(\alpha,\beta)}}} \left[-\frac{\alpha+\beta+2}{2}+\alpha+\beta+m+1\right] A_m,\\
  \theta_m&=&\frac{\beta-\alpha}{2}-\frac{\alpha+\beta+2}{2}B_m +m\left(\frac{\alpha-\beta}{\alpha+\beta+2m}-B_m\right)\!,\\
  b_m&=&\sqrt{\frac{g_{m+1}^{(\alpha,\beta)}}{g_m^{(\alpha,\beta)}}} \left(-\frac{\alpha+\beta+2}{2}-m\right)\!C_m.
\end{eqnarray*}
We require that $\eta_m=-b_{m-1}$ and $\theta_m=0$. However, before we look further into the above quantities, we note that necessarily $b_m<0$. Fortunately, this is an artefact of our choice of $k_m=0$: once we replace this with $k_m=m$, we obtain $b_m>0$, $m\in\BB{Z}_+$, as required. This requires long, yet simple algebra, best performed using a symbolic algebra package. It follows from \R{eq:1.8} that
\begin{eqnarray*}
  \sqrt{\frac{g_{m-1}^{(\alpha,\beta)}}{g_m^{(\alpha,\beta)}}} &=&\sqrt{\frac{m(\alpha+\beta+m)(\alpha+\beta+2m+1)}{(\alpha+m)(\beta+m)(\alpha+\beta+2m-1)}},\\
  \sqrt{\frac{g_{m+1}^{(\alpha,\beta)}}{g_m^{(\alpha,\beta)}}}&=&\sqrt{\frac{\alpha+m+1)(\beta+m+1)(\alpha+\beta+2m+1)}{(m+1)(\alpha+\beta+m+1)(\alpha+\beta+2m+3)}},
\end{eqnarray*}
hence all the conditions are satisfied for 
\begin{equation}
\label{eq:2.8}
b_m=\sqrt{\frac{(m+1)(\alpha+m+1)(\beta+m+1)(\alpha+\beta+m+1)}{(\alpha+\beta+2m+1)(\alpha+\beta+2m+3)}},\qquad m\in\BB{Z}_+.
\end{equation}

\setcounter{equation}{0}
\setcounter{figure}{0}
\section{A half-range setting}

In this section we take the model from the previous section a step further. Rather than requiring that our coefficients are equal to coefficients of a modified function in a single orthogonal polynomial basis, we consider two polynomial bases which are linked to the odd and even basis elements. {In principle this provides a much richer range of possibilities over the more restrictive setting of the previous section, but we will discuss in subsection 3.3 that we have not actually gained anything in the end.}

Specifically, suppose we have an orthonormal system $\Phi$ in $\CC{L}_2(\BB{R})$ such that $\varphi_m$ is an even function if $m$ is even and an odd function if $m$ is odd; let us call this an \emph{even-odd system}. Then we can re-express the expansion coefficients in the form
\begin{displaymath}
\hat{f}_{2m}=2\int_0^\infty \! f_\EE(x)\varphi_{2m}(x)\D x, \quad \hat{f}_{2m+1}=2\int_0^\infty \! f_\OO(x)\varphi_{2m+1}(x)\D x,\qquad m\in\BB{Z}_+,
\end{displaymath}
where
\begin{displaymath}
f_\EE(x)=\frac12[f(x)+f(-x)],\quad f_\OO(x)=\frac12[f(x)-f(-x)],\qquad x\in(0,\infty).
\end{displaymath}
In other words, we can `translate' the setting from the real line to $(0,\infty)$ and this forms the theme of this section. We suppose that the odd and even coefficients can each be expressed in a similar fashion to the previous section, but with separate expressions for each case, as follows. Let $\{r_m\}_{m\in\BB{Z}_+}$ and $\{s_m\}_{m\in\BB{Z}_+}$ be orthonormal polynomial systems with respect to the absolutely continuous measures $w_\EE(t)\!\D t$ and $w_\OO(t)\!\D t$, respectively, on $(-1,1)$. We wish to find all orthonormal, even-odd systems $\Phi$ with a tridiagonal, skew-symmetric differentiation matrix, i.e.
\begin{equation}
\label{eq:3.1}
\varphi_m'=-b_{m-1}\varphi_{m-1}+b_m\varphi_{m+1},\qquad m\in\BB{Z}_+,
\end{equation}
and such that the coefficients are equal to
\begin{eqnarray}
  \label{eq:3.2}
  \hat{f}_{2m}&=&\int_{-1}^1 \theta_\EE(t) f_\EE(h(t)) r_m(t) w_\EE(t) \D t,\\
  \nonumber
  \hat{f}_{2m+1}&=&\int_{-1}^1 \theta_\OO(t) f_\OO(h(t)) s_m(t) w_\OO(t) \D t,\qquad m\in\BB{Z}_+,
\end{eqnarray}
where $h : (-1,1) \to (0,\infty)$ and $\theta_\EE, \theta_\OO : (-1,1) \to \BB{R}$. 

We will make the following assumptions about $h$, $\theta_\EE$, $\theta_\OO$, $w_\EE$ and $w_\OO$. Just as in the previous section, we will \emph{deduce\/} more from these basic assumptions and our model.
\begin{itemize}
  \item $h$ maps onto the whole of $(0,\infty)$, is differentiable with $h'$ a measurable function, and $h'(t) > 0$ for all $t \in (-1,1)$. 
  \item This implies the existence of an inverse function $H : (0,\infty) \to (-1,1)$ which is differentiable with $h'(t)H'(h(t)) = 1$, which implies that $H'$ is also positive and measurable.
  \item $\theta_\EE$ is such that $t \mapsto \theta_\EE(t)\sqrt{w_\EE(t)/h'(t)} \in \CC{L}_\infty(\BB{R})$, and $\theta_\OO$ satisfies the analogous property. The motivation is exactly as in the previous section.
\end{itemize}

Changing variables to $x = h(t)$ yields
\begin{eqnarray}
  \label{eq:3.3}
  \hat{f}_{2m}&=&\int_{0}^\infty \Theta_\EE(x) f_\EE(x) r_m(H(x))\, \D x,\\
  \nonumber
  \hat{f}_{2m+1}&=&\int_{0}^\infty \Theta_\OO(x) f_\OO(x) s_m(H(x)) \, \D x,\qquad m\in\BB{Z}_+,
\end{eqnarray}
where $\Theta_\EE(x) = \theta_\EE(H(x))H'(x) w_\EE(H(x))$ and $\Theta_\OO(x) = \theta_\OO(H(x))H'(x) w_\OO(H(x))$. For this to hold for all $f \in \CC{L}_2(\BB{R})$, we must necessarily have the `half-range model',
\begin{eqnarray}
  \label{eq:3.4}
  \varphi_{2m}(x)&=&\Theta_\EE(x)r_m(H(x))\\
  \nonumber
  \varphi_{2m+1}(x)&=&\Theta_\OO(x)s_m(H(x)),\qquad m\in\BB{Z}_+.
\end{eqnarray}
We extend $H$ to the whole of $\BB{R}$ by setting $H(-x) = H(x)$. Since $\varphi_{0}$ is an even function, we must have that $\Theta_\EE$ is even function and likewise $\Theta_\OO$ is an odd function. Note that $\varphi_m$ is an infinitely differentiable function for all $m\in\BB{Z}_+$ because it is the Fourier transform of a superalgebraically decaying function by equation \eqref{eq:1.5}. Therefore, $\varphi_1(0) = 0$ in order to have oddness. It follows that $\Theta_\OO(0) = 0$. Allow us to place one more assumption into the mix: $\Theta_\EE(0) \neq 0$. Otherwise all basis functions vanish at the origin, rendering it clearly unsuitable for {the} approximation of functions which are in general \emph{nonzero\/} at the origin.

The rest of the section is devoted to proving the following.

\begin{theorem}\label{thm:alpha}
  All the systems \R{eq:3.4} which are orthonormal in $\CC{L}_2(\BB{R})$ and possess a tridiagonal, skew-symmetric differentiation matrix are, up to a linear change of variables,
  \begin{eqnarray}
    \label{eq:3.5}
    \varphi_{2m}(x)&=& \frac{2^{(2\alpha+1)/4}}{\sqrt{g_m^{(\alpha,-\frac12)}}}\frac{1} {\cosh^{1+\alpha}\!x} \PP_m^{(\alpha,-\frac12)}\!\left(1-\frac{2}{\cosh^2x}\right)\!,\\
    \nonumber
    \varphi_{2m+1}(x)&=&-\frac{2^{(2\alpha+3)/4}}{\sqrt{g_m^{(\alpha,\frac12)}}}  \frac{\sinh x}{\cosh^{2+\alpha}\!x} \PP_m^{(\alpha,\frac12)}\!\left(1-\frac{2}{\cosh^2x}\right)\!,\qquad m\in\BB{Z}_+,
  \end{eqnarray}
  for any $\alpha>-1$.
\end{theorem}

We show that these bases are equal to the bases in Theorem \ref{thm:alphabeta} with $\beta = \alpha$, something which is far from obvious at first sight. All discussions of mathematical properties of these functions such as completeness in $\CC{L}_2(\BB{R})$ can therefore be derived from properties of the functions in Theorem \ref{thm:alphabeta}.

\subsection{Necessary conditions}

The first condition to explore is orthonormality. Since the $\varphi_m$s have the parity of $m$ on the real line, we have
\begin{displaymath}
\int_{-\infty}^\infty \varphi_{2m}(x)\varphi_{2n+1}(x)\D x=0,\qquad m,n\in\BB{Z}_+
\end{displaymath}
and need to check orthogonality only within each set. Changing variables, it follows from \R{eq:3.4} that
\begin{displaymath}
\int_{-\infty}^\infty \varphi_{2m}(x)\varphi_{2n}(x)\D x=2 \int_0^\infty \varphi_{2m}(x)\varphi_{2n}(x)\D x=2\int_{-1}^1 \Theta_\EE^2(h(t))h'(t) r_m(t)r_n(t)\D t.
\end{displaymath}
Since $\{r_m\}_{m\in\BB{Z}_+}$ is an orthonormal set with respect to $w_\EE(t)\!\D t$, we must have
\begin{equation}
\label{eq:3.6}
w_\EE(t)=2\Theta_\EE^2(h(t))h'(t),\qquad t\in(-1,1),
\end{equation}
and, by the same token,
\begin{equation}
\label{eq:3.7}
w_\OO(t)=2\Theta_\OO^2(h(t))h'(t),\qquad t\in(-1,1).
\end{equation}
Note that, by the monotonicity of $h$, both weight functions are nonnegative, as required. These expressions can be inverted: changing  back to $x$,
\begin{equation}
\label{eq:3.8}
\Theta_\EE(x)=\sqrt{\frac12 H'(x)w_\EE(H(x))},\quad \Theta_\OO(x)=\pm\sqrt{\frac12 H'(x)w_\OO(H(x))},\qquad x\in(0,\infty).
\end{equation}
Observe that we need to be very careful in our choice of sign. As things stand, we allow for both options. 

Our next, considerably more challenging task is to ensure the existence of a differentiation matrix of the correct form. Recall that both $\Theta_\EE$ and $\Theta_\OO$ are smooth functions on $\BB{R}$. Substituting \eqref{eq:3.4} into \eqref{eq:3.1}, we have 
\begin{eqnarray}
  \label{eqn:halfrangeODE}
  \Theta_\EE'r_m(H)+\Theta_\EE H'r_m'(H)&=&\Theta_\OO[b_{2m-1}s_{m-1}(H)+b_{2m}s_m(H)],\\
  \nonumber
  \Theta_\OO's_m(H)+\Theta_\OO H's_m'(H)&=&\Theta_\EE[-b_{2m-1}r_m(H)+b_{2m+1}r_{m+1}(H)],\qquad m\in\BB{Z}_+.
\end{eqnarray}
Setting $m=0$ yields
\begin{eqnarray}
  \label{eqn:ThetaEEprime}
  \Theta_\EE'(x) &=&\Theta_\OO(x) b_{0}s_0/r_0,\\
  \nonumber
  \Theta_\OO'(x) &=&\Theta_\EE(x) b_{1}r_{1}(H(x))/s_0.
\end{eqnarray}
Substituting these back into \eqref{eqn:halfrangeODE} and changing variables back to $t$, we have,
\begin{eqnarray}
  \label{eq:3.9}
  r_m'(t)&=&h'(t)\sqrt{\frac{w_\OO(t)}{w_\EE(t)}} A_m(t) \\
  \label{eq:3.10}
  s_m'(t)&=&h'(t)\sqrt{\frac{w_\EE(t)}{w_\OO(t)}} B_m(t),\qquad m\in\BB{N},
\end{eqnarray}
where
\begin{eqnarray*}
  A_m(t)&=&-b_{2m-1}s_{m-1}(t)+b_{2m}s_m(t)-\frac{b_0s_0}{r_0} r_m(t),\\
  B_m(t)&=&-b_{2m}r_m(t)+b_{2m+1}r_{m+1}(t)+\frac{b_0 r_0-b_1r_1(t)}{s_0} s_m(t),\qquad m\in\BB{N}.
\end{eqnarray*}
The way forward rests upon the observation that the left-hand sides of both \eqref{eq:3.9} and \eqref{eq:3.10} are $(m-1)$-degree polynomials, and this places important constraints upon their right-hand sides. This is similar to the analysis of Section~2 yet considerably more complicated.

Setting $m=1$, taking products of equations \eqref{eq:3.9} and \eqref{eq:3.10} and performing some simple algebra, we obtain,
\begin{displaymath}
h'(t) = \sqrt{\frac{r_1' s_1'}{A_1(t)B_1(t)}} \qquad\text{and}\qquad \frac{w_\OO(t)}{w_\EE(t)} = \frac{r_1' B_1(t)}{s_1' A_1(t)}.
\end{displaymath}
Since $r_1$, $s_1$, and $A_1$ are polynomials of degree 1 and $B_1$ is a polynomial of degree 2, there exist constants $a,b,c, \gamma_1,\gamma_2$ (which are real except that $b$ and $c$ may be complex conjugates) such that
\begin{equation}
h'(t) = \frac{\gamma_1}{\sqrt{(1-at)(1-bt)(1-ct)}} \quad\text{and}\quad \frac{w_\OO(t)}{w_\EE(t)} = \gamma_2\frac{(1-bt)(1-ct)}{1-at}.
\end{equation}
Substituting the ratio of $w_\OO(t)$ and $w_\EE(t)$ for $t \in (-1,1)$ into equations \eqref{eq:3.6} and \eqref{eq:3.7}, we find
\begin{equation}
\label{eq:3.11}
\Theta_\OO(h(t)) = \sqrt{\gamma_2\frac{(1-bt)(1-ct)}{1-at}} \,\Theta_\EE(h(t)).
\end{equation}
Since $h(-1) = 0$, $\Theta_\OO(0) = 0$ and $\Theta_\EE(0) \neq 0$ we deduce (without loss of generality) that $c = -1$. Since $\lim_{t\to 1}h(t) = +\infty$, we must have, integrating $h'(t)$,
\begin{displaymath}
\lim_{t \to 1} \int_{-1}^t \frac{1}{\sqrt{(1-as)(1-bs)(1+s)}} \, \D s = +\infty.
\end{displaymath}
This is only possible if $a = b = 1$. Therefore,
\begin{displaymath}
h'(t) = \frac{\gamma_1}{(1-t)\sqrt{1+t}}, \qquad w_\OO(t) = \gamma_2 (1+t) w_\EE(t).
\end{displaymath}
The formula for $h'(t)$ is readily integrated using the substitution $\tanh u = \sqrt{(1+s)/2}$, giving
\begin{displaymath}
h(t) = \sqrt{2}\gamma_1\arctanh\sqrt{\frac{1+t}{2}}.
\end{displaymath}
Inverting, and ignoring a linear change of variables in $x$, we have whittled our way down to a single option,
\begin{displaymath}
H(x) = 1 - \frac{2}{\cosh^2x}.
\end{displaymath}
Now let us find $\Theta_\EE(x)$. We know that $\Theta_\EE'(x) =  \Theta_\OO(x) b_0s_0r_0$ by equation \eqref{eqn:ThetaEEprime} and $\Theta_\OO(x) = \sqrt{\gamma_2(1+H(x))}\Theta_\EE(x)$ by equation \eqref{eq:3.11}. Hence,
\begin{displaymath}
\frac{\D}{\D t} \Theta_\EE(h(t)) = \Theta_\EE'(h(t)) h'(t) = \frac{\Theta_E(h(t))}{1-t} \sqrt{\gamma_2}\gamma_1b_0s_0/r_0.
\end{displaymath}
In consequence, $\Theta_\EE(h(t)) \propto (1-t)^{\frac{\alpha+1}{2}}$ for some $\alpha > -1$, and, converting to the $x$ variable, we have
\begin{displaymath}
\Theta_\EE(x) \propto \frac{1}{\cosh^{\alpha+1}(x)}, \qquad \Theta_\OO(x) \propto \frac{\sinh(x)}{\cosh^{2+\alpha}(x)}.
\end{displaymath}
Equations \eqref{eq:3.6} and \eqref{eq:3.7} give us the weights,
\begin{displaymath}
w_\EE(t) = (1-t)^\alpha (1+t)^{-\frac12}, \qquad w_\OO(t) = (1-t)^\alpha (1+t)^{\frac12},
\end{displaymath}
where without loss of generality $w_\EE(0) = w_\OO(0) = 1$.

\subsection{Sufficient conditions}

As things stand, we have identified one -- and just one -- one-parameter family of weights $\{w_\EE,w_\OO\}$ for which we might be able to obtain an orthonormal system \eqref{eq:3.4} with a tridiagonal skew-symmetric differentiation matrix:
\begin{eqnarray*}
  &&w_\EE(t)=(1-t)^\alpha(1+t)^{-\frac12},\quad w_\OO(t)=(1-t)^\alpha(1+t)^{\frac12},\qquad t\in(-1,1),\quad \alpha>-1,\\
  &&r_m(t)=\frac{1}{\sqrt{g_m^{(\alpha,-\frac12)}}}\PP_m^{(\alpha,-\frac12)}(t),\quad s_m(t)=\frac{1}{\sqrt{g_m^{(\alpha,\frac12)}}}\PP_m^{(\alpha,\frac12)}(t),\qquad m\in\BB{Z}_+,\\
  &&h(t)=\tanh^{-1}\sqrt{\frac{1+t}{2}},\qquad H(x)=1-\frac{2}{\cosh^2x},\\
  &&\Theta_\EE(x)=2^{\frac14+\frac{\alpha}{2}} \frac{1}{\cosh^{1+\alpha}x},\qquad \Theta_\OO(x)=-2^{\frac34+\frac{\alpha}{2}} \frac{\sinh x}{\cosh^{2+\alpha}x},
\end{eqnarray*}
where we have used \R{eq:3.8} (with a minus sign) to determine $\Theta_\EE$ and $\Theta_\OO$. The question is, do the systems here actually satisfy our requirements for all $\alpha > -1$? The following subsection answers this question in the affirmative by relating these functions to the full-range systems from Section 2.

\subsection{The connection to full-range systems}

Further investigation of these half-range systems leads one to the conclusion the half-range systems of Theorem \ref{thm:alpha} are a special case of full-range systems of Theorem \ref{thm:alphabeta} with $\beta = \alpha$. The formul\ae{} look completely different, but as we will now show, they are identical.

{ For half-range functions
  \begin{displaymath}
  \varphi_{2m}(x)=\frac{2^{\frac{2\alpha+1}{4}}}{\sqrt{g_m^{(\alpha,-\frac12)}}} \frac{1}{\cosh^{\alpha+1}\!x} \PP_m^{(\alpha,-\frac12)}\!\left(1-\frac{2}{\cosh^2x}\right)\!,
  \end{displaymath}
  while for full range (with $\beta$ set to $\alpha$),
  \begin{eqnarray*}
    \varphi_{2m}(x)&=&\frac{1}{\sqrt{g_{2m}^{(\alpha,\alpha)}}} (1-\tanh^2x)^{\frac{\alpha+1}{2}}\CC{P}_{2m}^{(\alpha,\alpha)}(\tanh x).
  \end{eqnarray*}
  
  Now, by \cite[18.7.13]{DLMF}, $\PP_{2m}^{(\alpha,\alpha)}(X) \propto \PP_{m}^{(\alpha,-\frac12)}(2X^2-1)$. Using this with $X = \tanh x $, along with the identity $\tanh^2 x = 1- \mathrm{sech}^2 x$ it follows readily that $\PP_{2m}^{(\alpha,\alpha)}(\tanh x ) \propto \PP_{m}^{(\alpha,-\frac12)}(1 - 2/\cosh^2 x)$. In addition, we know that $(1-\tanh^2 x)^{\frac{\alpha+1}{2}} \propto \mathrm{sech}^2 x $. Combining the identities discussed in this paragraph, we arrive at the proportionality statement,
  
  \begin{equation*}
  \frac{1}{\cosh^{\alpha+1}\!x} \PP_m^{(\alpha,-\frac12)}\!\left(1-\frac{2}{\cosh^2x}\right) \propto (1-\tanh^2x)^{\frac{\alpha+1}{2}}\CC{P}_{2m}^{(\alpha,\alpha)}(\tanh x).
  \end{equation*}
  
  The fact that the constants which depend on $\alpha$ but not $x$ are uniquely determined to ensure that $\varphi_{2m}$ has $\CC{L}_2(\BB{R})$ norm equal to 1 proves that the two expressions for $\varphi_{2m}$ given above are identical.
  
  Likewise, in a half-range formalism,
  \begin{displaymath}
  \varphi_{2m+1}(x)=-\frac{2^{\frac\alpha2+\frac34}}{\sqrt{g_m^{(\alpha,\frac12)}}} \frac{\sinh x}{\cosh^{2+\alpha}x} \PP_m^{(\alpha,\frac12)}\!\left(1-\frac{2}{\cosh^2x}\right)\!,
  \end{displaymath}
  while the full-range expression is 
  \begin{eqnarray*}
    \varphi_{2m+1}(x)&=&\frac{(-1)^m}{\sqrt{g_{2m+1}^{(\alpha,\alpha)}}} (1-\tanh^2 x)^{\frac{\alpha+1}{2}} \PP_{2m+1}^{(\alpha,\alpha)}(\tanh x).
  \end{eqnarray*}
  This time, we have by \cite[18.7.14]{DLMF} that $\PP_{2m+1}^{(\alpha,\alpha)}(X) \propto X\PP_{m}^{(\alpha,\frac12)}(2X^2-1)$. Using this with $X = \tanh x $, along with the identity $\tanh^2 x = 1- \mathrm{sech}^2 x$ it follows readily that $\PP_{2m+1}^{(\alpha,\alpha)}(\tanh x) \propto \tanh x \PP_{m}^{(\alpha,\frac12)}(1 - 2/\cosh^2 x)$. In addition, we know that $(1-\tanh^2 x )^{\frac{\alpha+1}{2}} \propto \mathrm{sech}^2 x$. Combining the identities discussed in this paragraph, we arrive at the proportionality statement,
  \begin{equation*}
  \frac{\sinh x}{\cosh^{2+\alpha}\!x} \PP_m^{(\alpha,\frac12)}\!\left(1-\frac{2}{\cosh^2x}\right) \propto (1-\tanh^2x)^{\frac{\alpha+1}{2}}\CC{P}_{2m+1}^{(\alpha,\alpha)}(\tanh x).
  \end{equation*}
  
  The fact that the constants which depend on $\alpha$ but not $x$ are uniquely determined to ensure that $\varphi_{2m+1}$ has $\CC{L}_2(\BB{R})$ norm equal to 1 proves that the two expressions for $\varphi_{2m+1}$ given above are identical.
  
  Theorem 2 therefore provides us with no new systems over the systems given by Theorem 1. This leads us to ask the question, could we have deduced Theorem 2 from Theorem 1, instead of conducting the full derivation of subsection 3.2? We believe the answer is no, but in principle if one could show that the model in equation (3.4) necessarily reduces to the model in equation (2.4), then Theorem 2 would indeed follow directly from Theorem 1.}

\setcounter{equation}{0}
\setcounter{figure}{0}
\section{The tanh-Jacobi functions}\label{sec:tanh-Jacobifunctions}

In Section 2 we identified the following orthonormal bases $\Phi^{(\alpha,\beta)} = \{\varphi^{(\alpha,\beta)}_m\}_{m\in\BB{Z}_+}$ for $\alpha,\beta > -1$,
\begin{displaymath}
\varphi^{(\alpha,\beta)}_m(x)=\frac{(-1)^m}{\sqrt{g_m^{(\alpha,\beta)}}} (1-\tanh x)^{\frac{\alpha+1}{2}}(1+\tanh x)^{\frac{\beta+1}{2}} \PP_m^{(\alpha,\beta)}(\tanh x),\qquad m\in\BB{Z}_+.
\end{displaymath}
{We call these function the \emph{tanh-Jacobi functions}, for obvious reasons.}  These orthonormal bases have a tridiagonal, irreducible, skew-symmetric differentiation matrix as in equation \eqref{eq:1.3}, with
\begin{displaymath}
b_m^{(\alpha,\beta)}=\sqrt{\frac{(m+1)(\alpha+m+1)(\beta+m+1)(\alpha+\beta+m+1)}{(\alpha+\beta+2m+1)(\alpha+\beta+2m+3)}},\qquad m\in\BB{Z}_+.
\end{displaymath}

In Section 3 we showed that in the special case where $\beta = \alpha$, there is an alternative, equivalent expression which separates the functions into odd and even parts. For all $\alpha > -1$, let
\begin{eqnarray*}
  \varphi^\alpha_{2m}(x)&=& \frac{2^{(2\alpha+1)/4}}{\sqrt{g_m^{(\alpha,-\frac12)}}}\frac{1} {\cosh^{1+\alpha}\!x} \PP_m^{(\alpha,-\frac12)}\!\left(1-\frac{2}{\cosh^2x}\right)\!,\\
  \varphi^\alpha_{2m+1}(x)&=&-\frac{2^{(2\alpha+3)/4}}{\sqrt{g_m^{(\alpha,\frac12)}}}  \frac{\sinh x}{\cosh^{2+\alpha}\!x} \PP_m^{(\alpha,\frac12)}\!\left(1-\frac{2}{\cosh^2x}\right)\!,\qquad m\in\BB{Z}_+.
\end{eqnarray*}
Then $\varphi^\alpha_m = \varphi^{(\alpha,\alpha)}_m$ for all $m \in \BB{Z}_+$, at least mathematically speaking. For computation of coefficients, this basis is different, as is discussed below.

\begin{figure}[htb!]
  \centering
  \includegraphics[width=.32\textwidth]{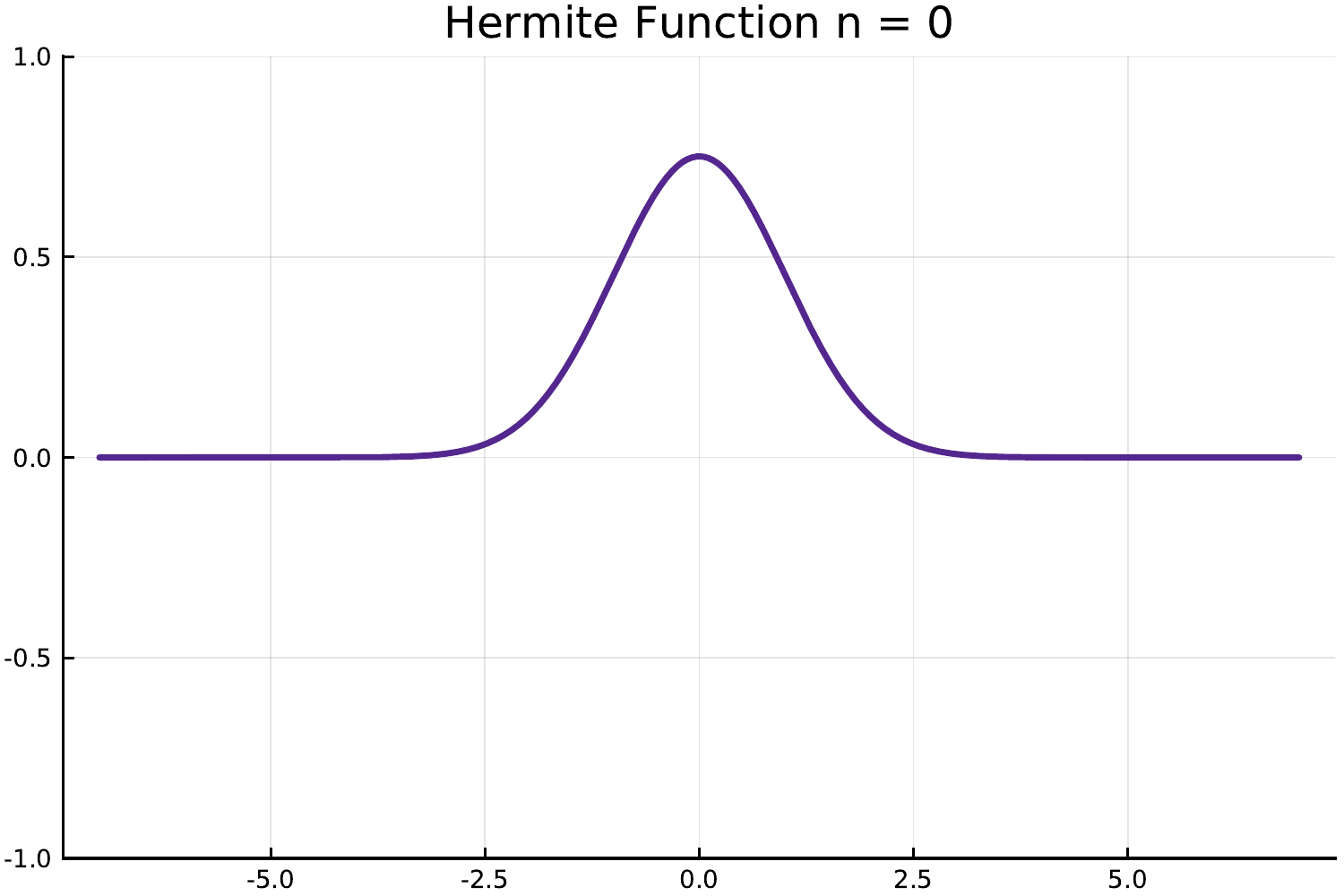}\,
  \includegraphics[width=.32\textwidth]{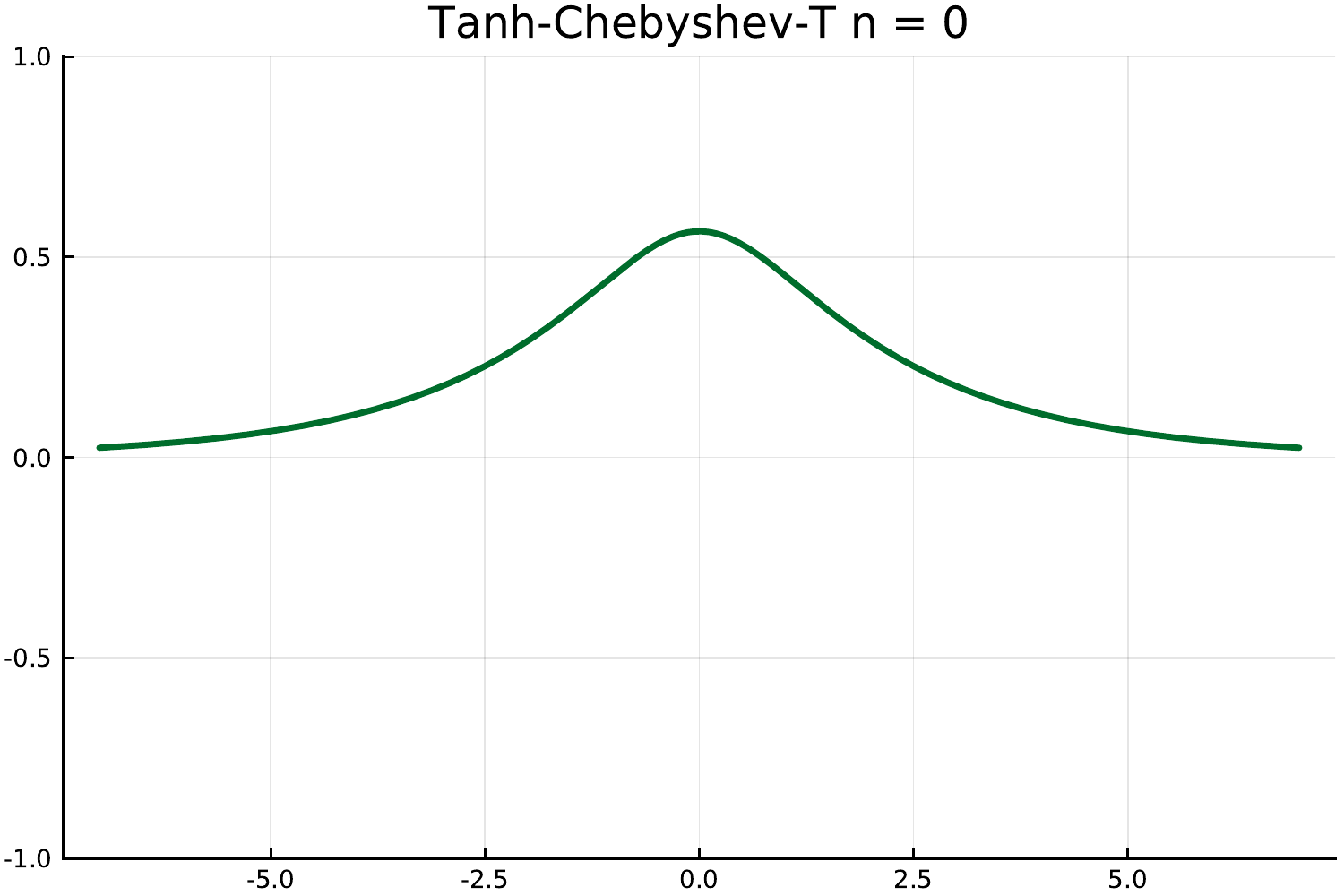}\,
  \includegraphics[width=.32\textwidth]{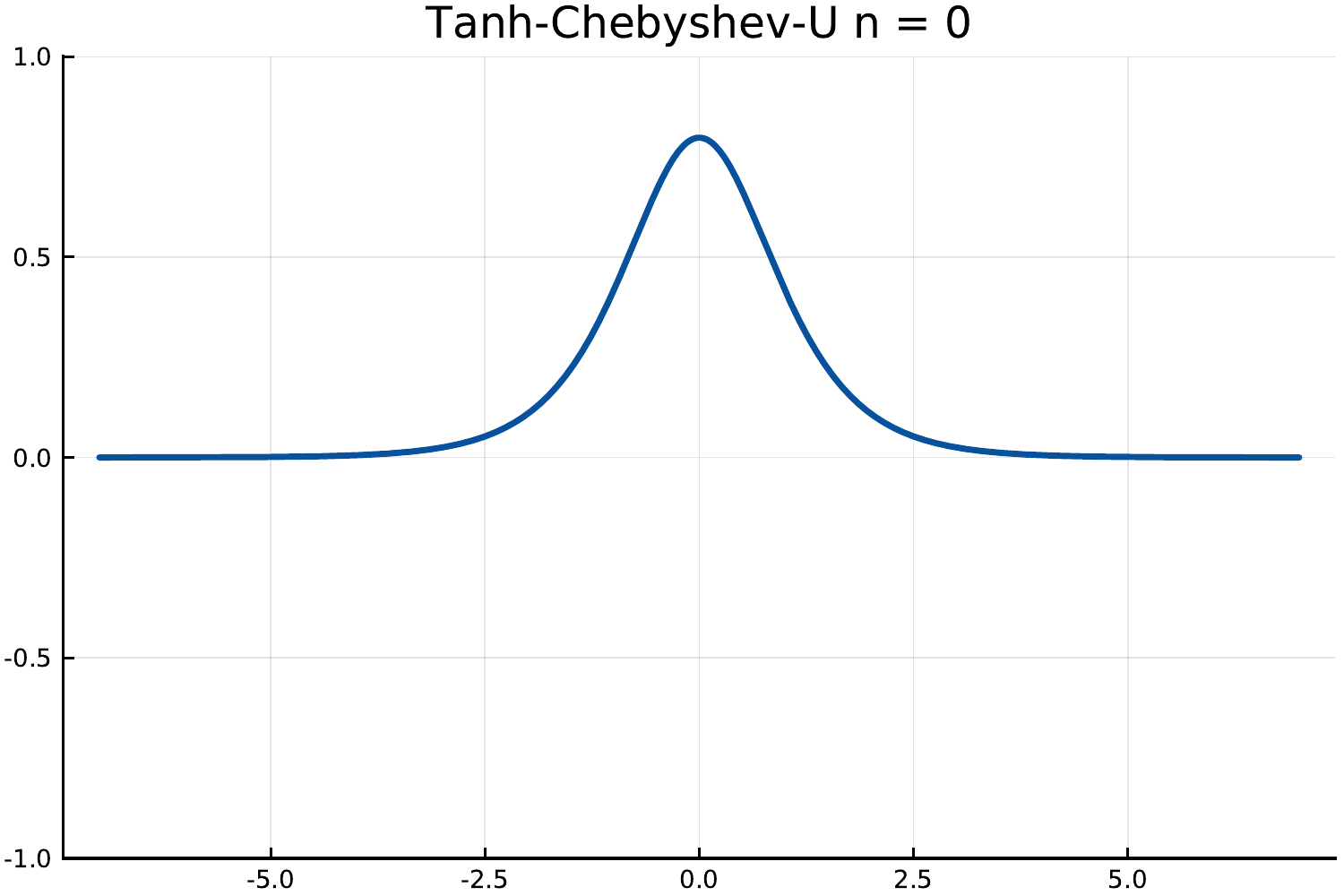}\\
  \vskip10pt
  \includegraphics[width=.32\textwidth]{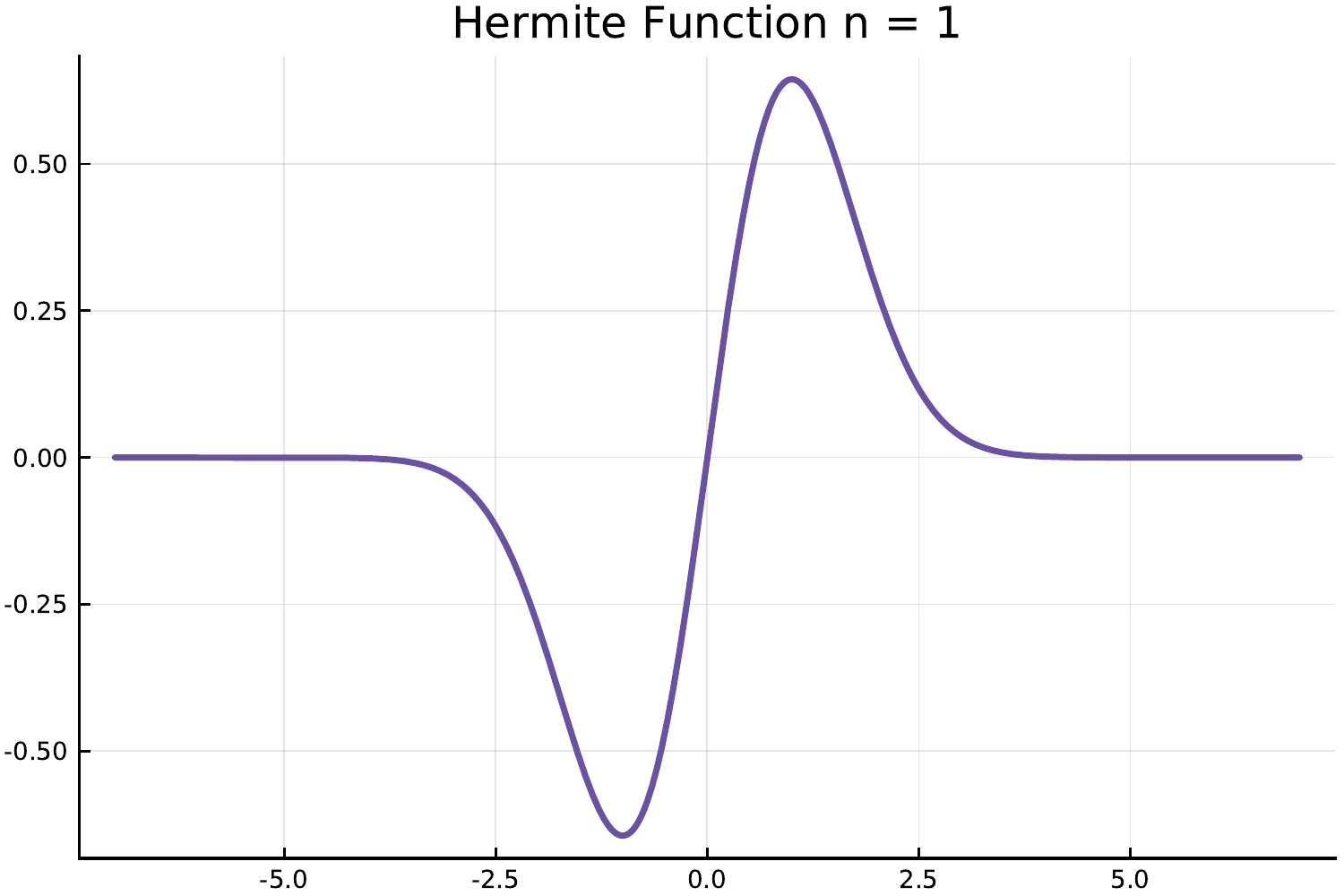}\,
  \includegraphics[width=.32\textwidth]{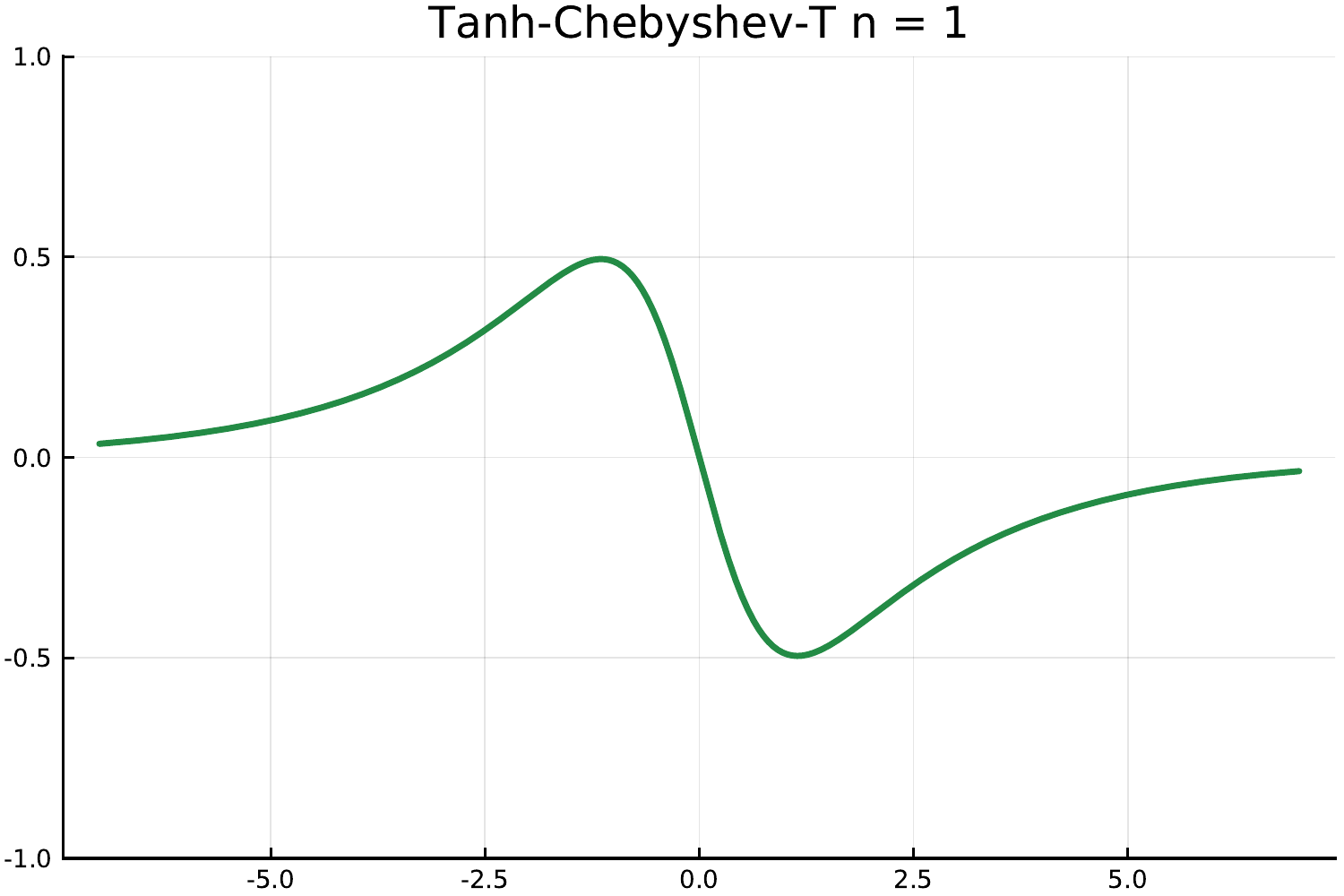}\,
  \includegraphics[width=.32\textwidth]{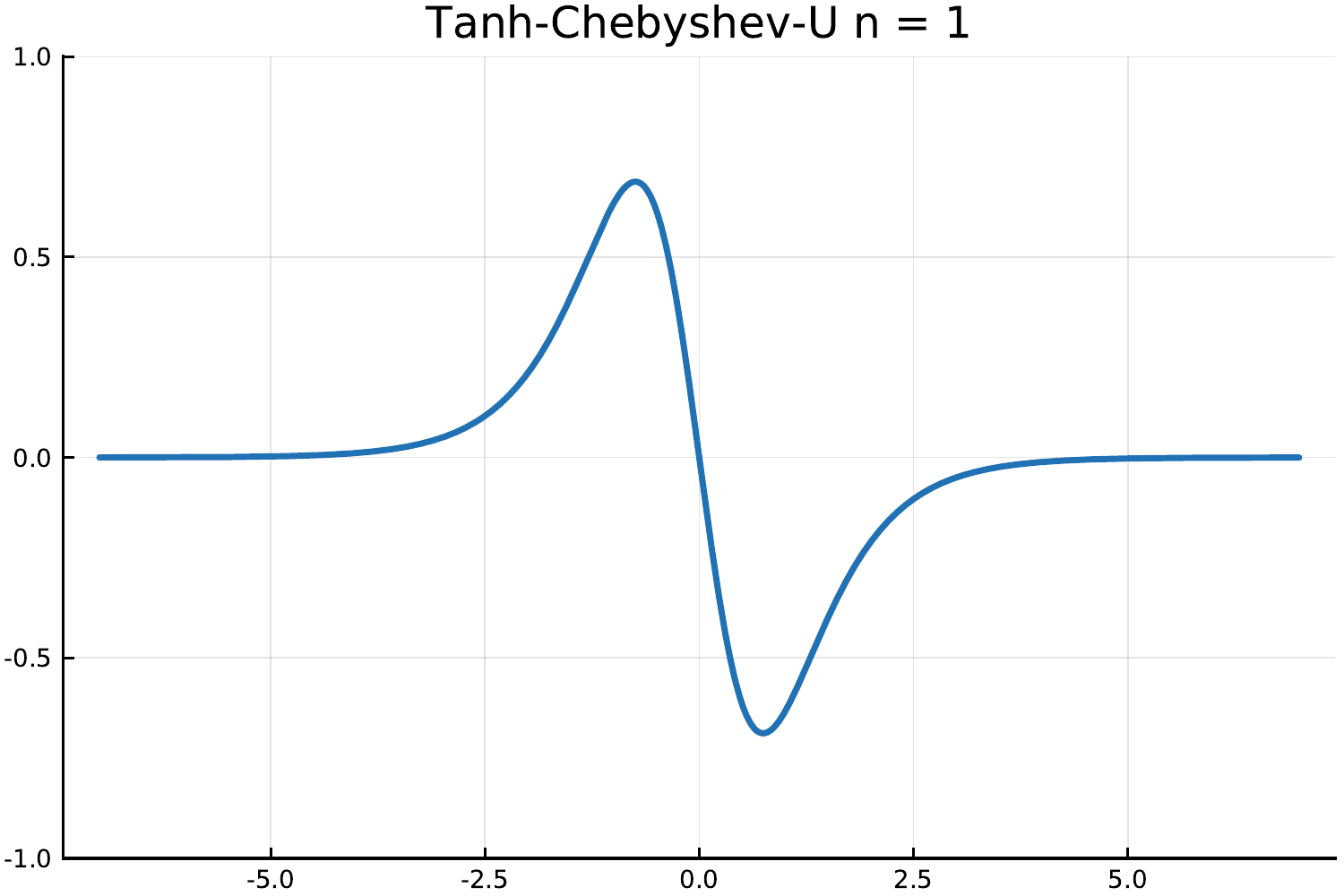}\\
  \vskip10pt
  \includegraphics[width=.32\textwidth]{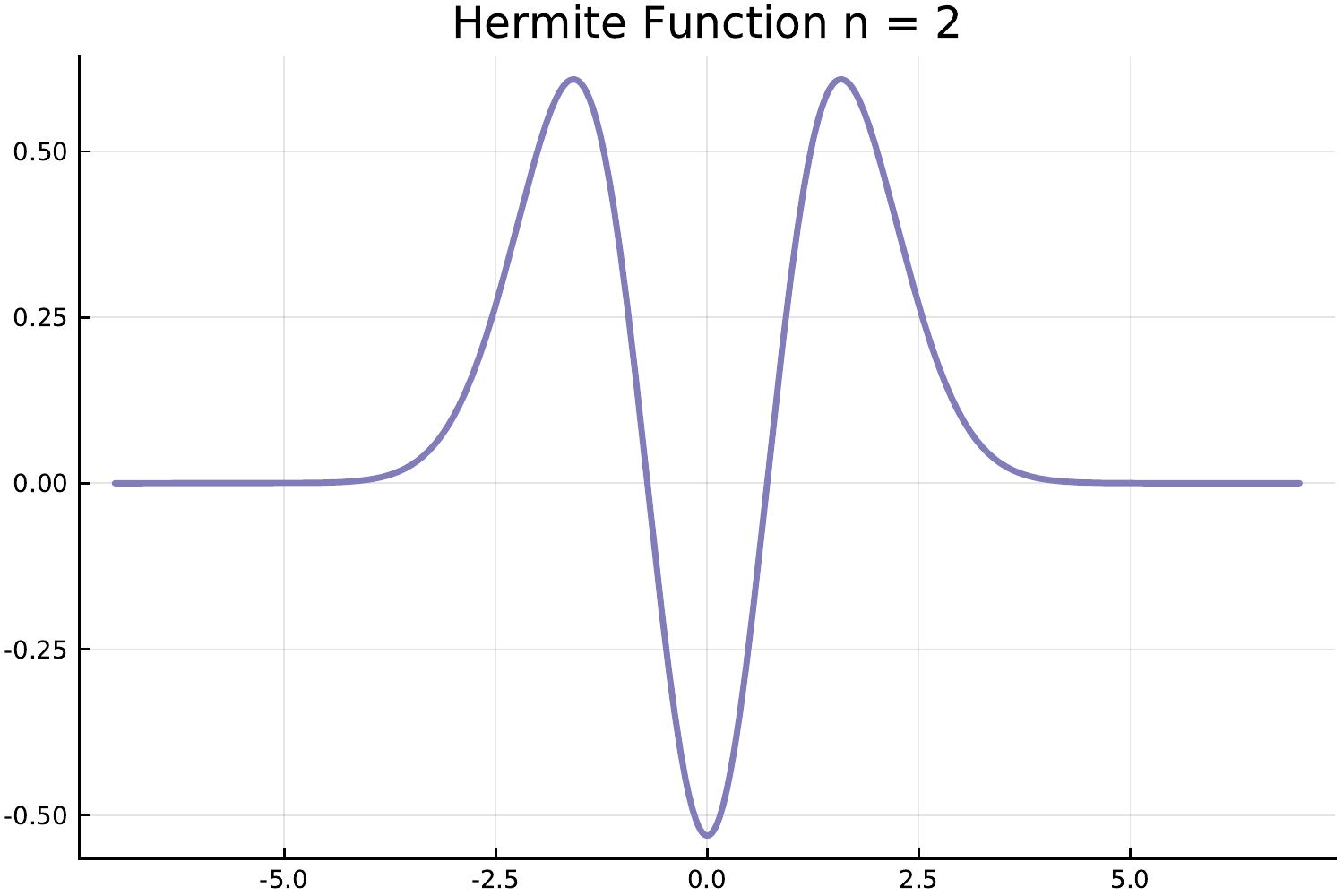}\,
  \includegraphics[width=.32\textwidth]{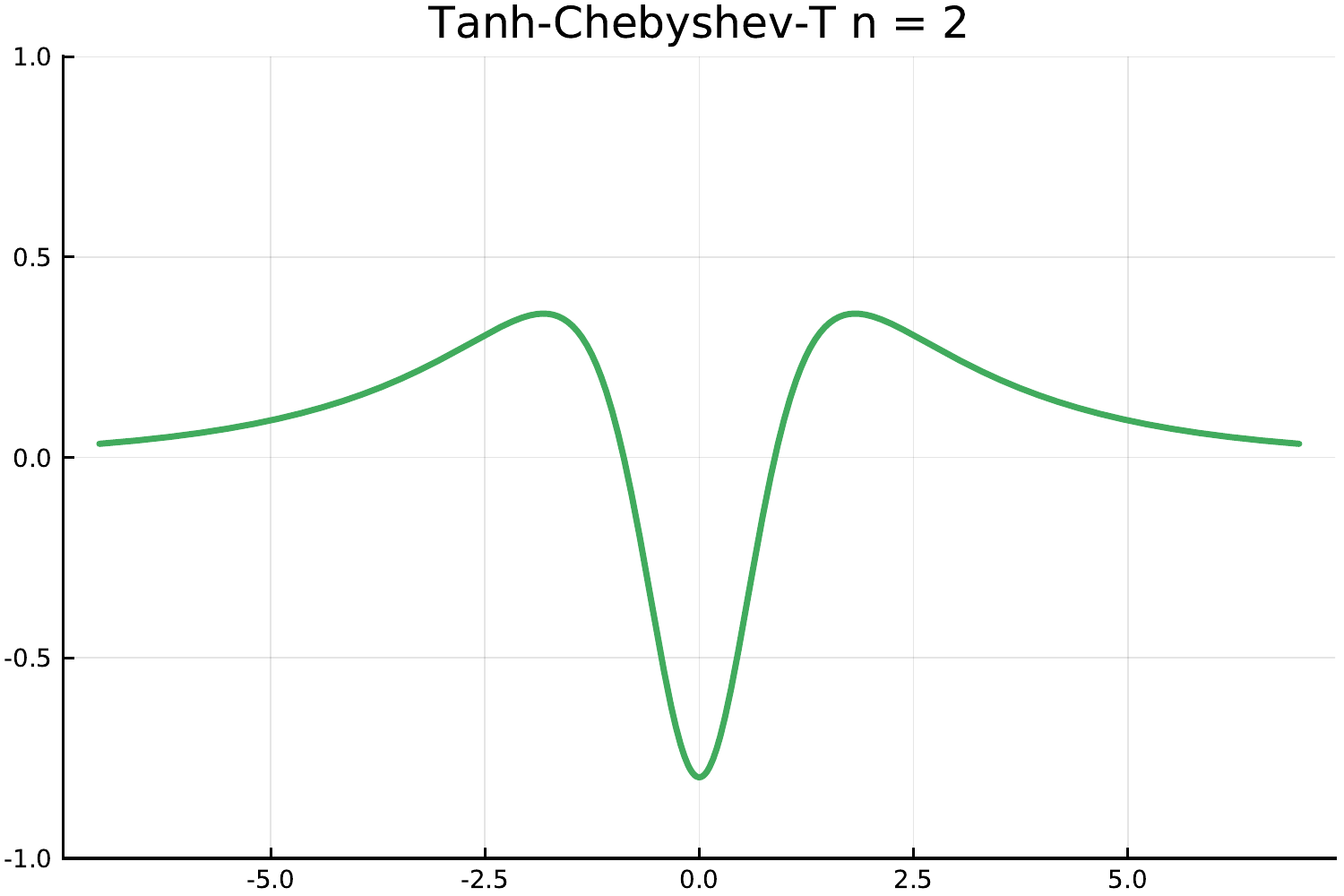}\,
  \includegraphics[width=.32\textwidth]{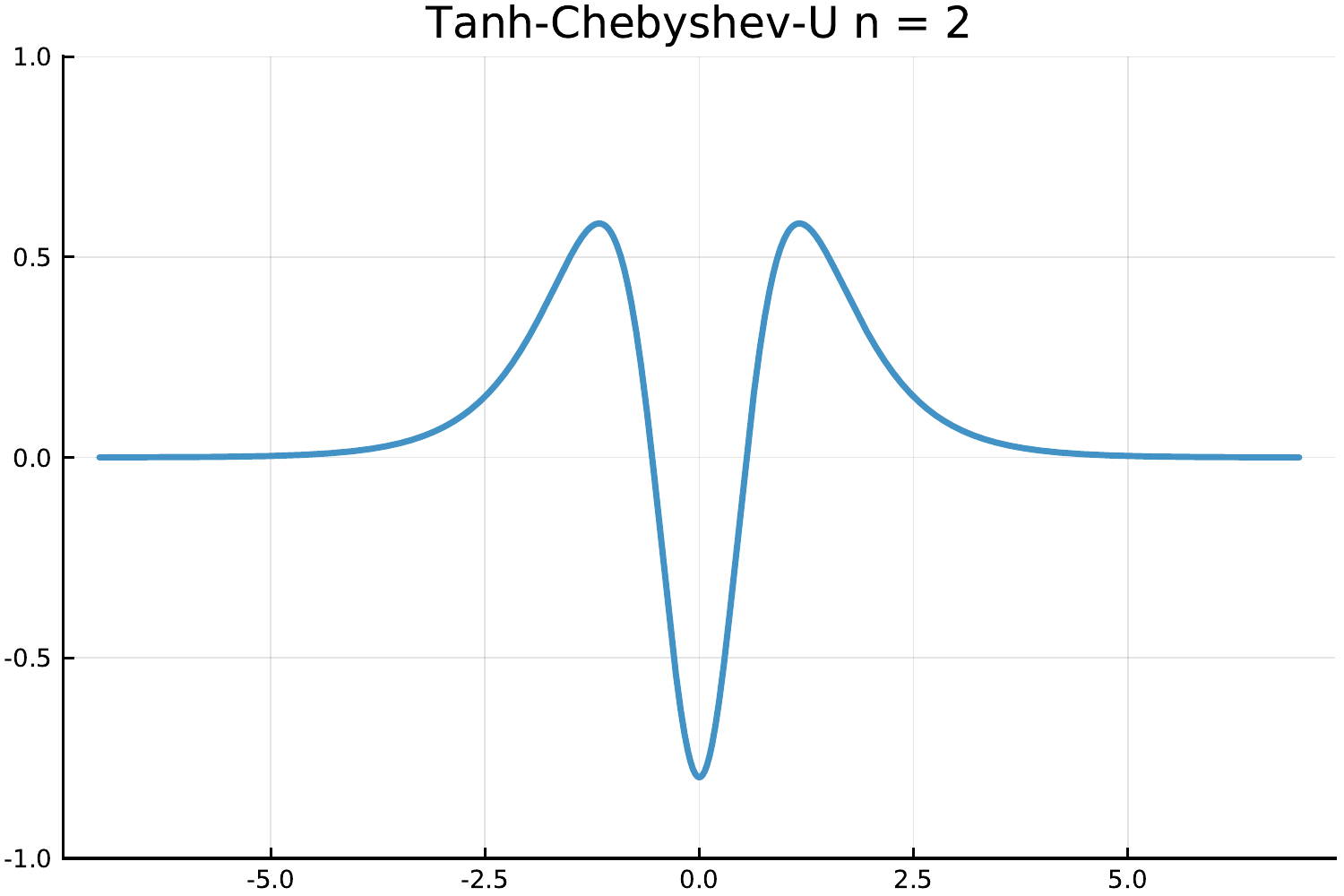}\\
  \vskip10pt
  \includegraphics[width=.32\textwidth]{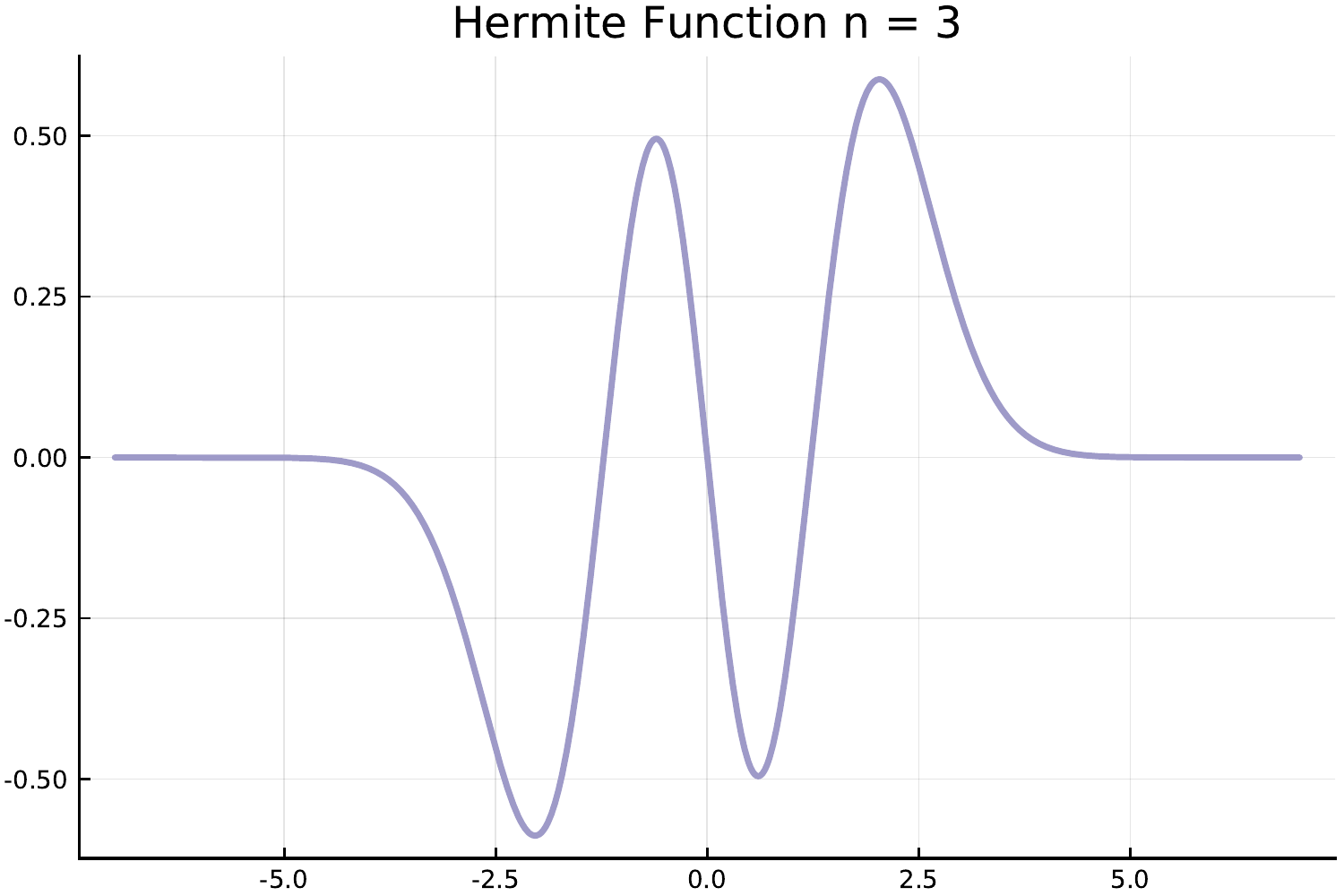}\, 
  \includegraphics[width=.32\textwidth]{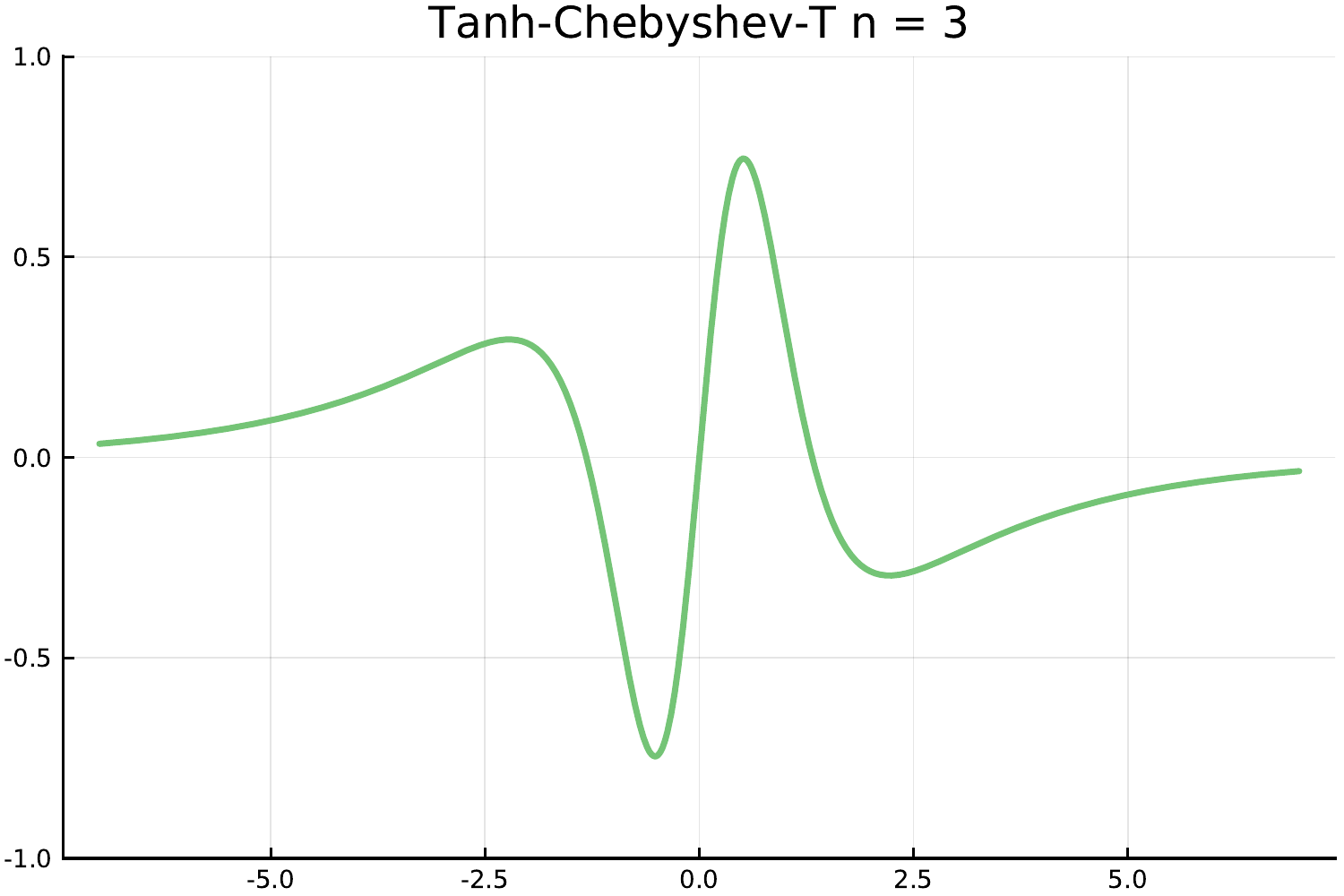}\,
  \includegraphics[width=.32\textwidth]{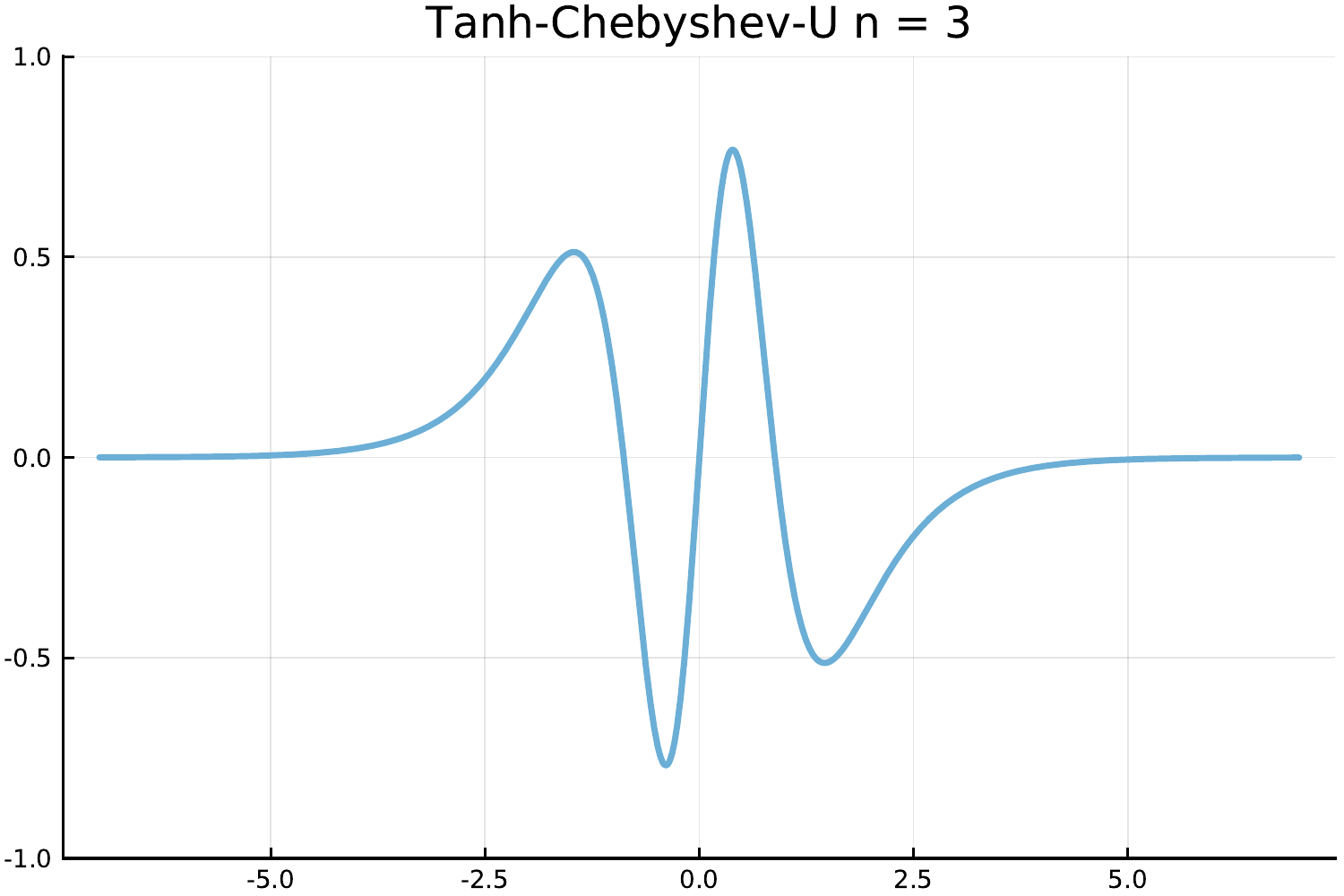}\\
  \vskip10pt
  \includegraphics[width=.32\textwidth]{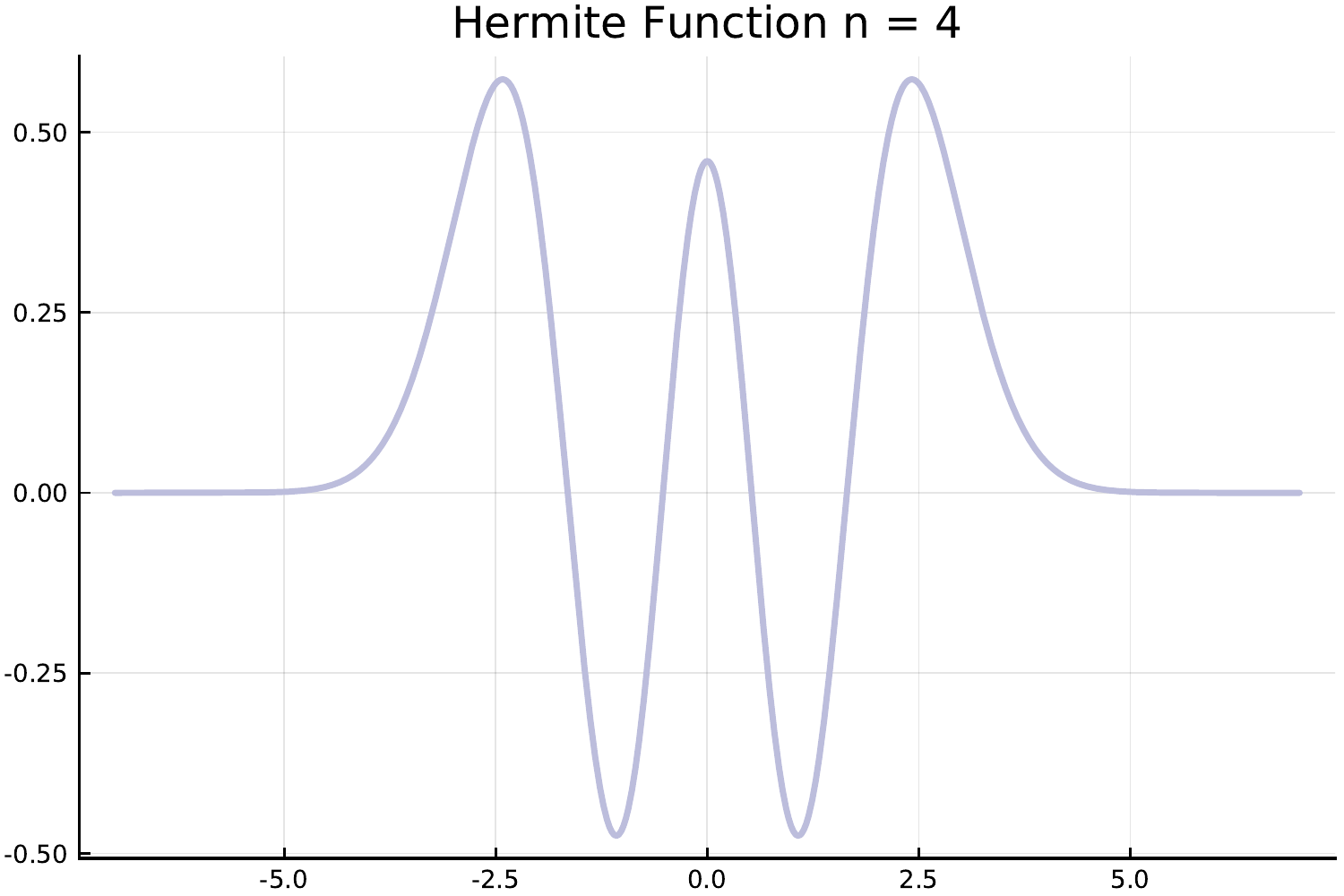}\,
  \includegraphics[width=.32\textwidth]{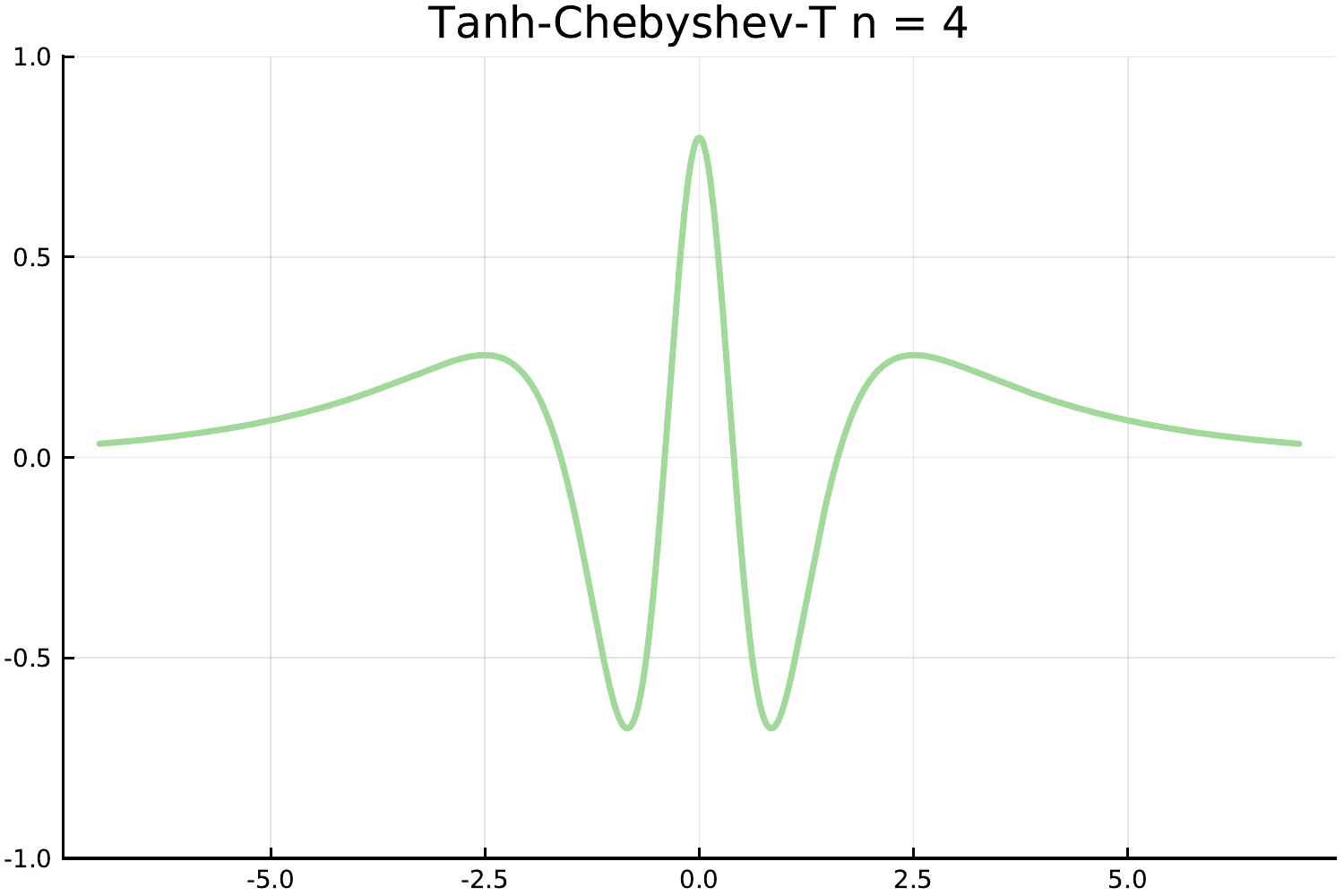}\,
  \includegraphics[width=.32\textwidth]{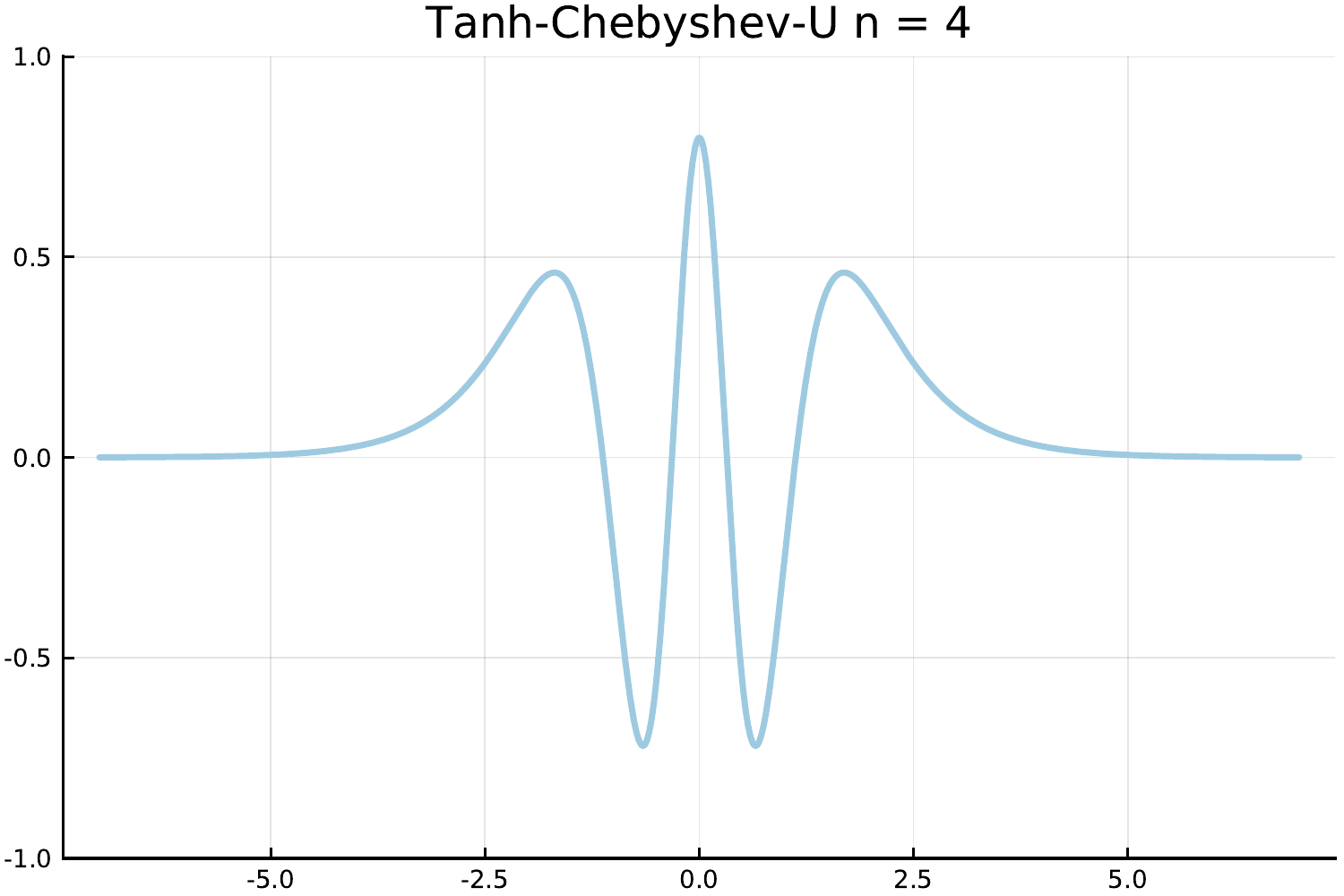}
  
  \caption{{The first five basis functions for the Hermite functions (left, purple), tanh-Chebyshev-T functions (center, green), and tanh-Chebyshev-U functions (right, blue). The top, darkest coloured line represents $m=0$, progressing to the bottom, lighter coloured line representing $m=4$.}}
\end{figure}

\subsection{Expansion coefficients}

Using the change of variables $t = \tanh x$, the expansion coefficients of a function $f \in \CC{L}_2(\BB{R})$ in the orthonormal basis $\Phi^{(\alpha,\beta)}$ can be expressed as

\begin{equation}
\label{eq:4.1}
\hat{f}_m = \frac{(-1)^m}{\sqrt{g_m^{(\alpha,\beta)}}} \int_{-1}^1 \frac{f(\arctanh t)}{(1-t)^{\frac{\alpha+1}{2}}(1+t)^{\frac{\beta+1}{2}}} \PP_m^{(\alpha,\beta)}(t) (1-t)^\alpha (1+t)^\beta \, \D t,
\end{equation}
which are the Jacobi coefficients of the modified function
\begin{displaymath}
F(t) = \frac{f(\arctanh t)}{(1-t)^{\frac{\alpha+1}{2}}(1+t)^{\frac{\beta+1}{2}}},
\end{displaymath}
with a diagonal scaling. The regularity of $F$ determines the convergence of the coefficients, since Jacobi polynomials bases have spectral convergence properties (this is an elementary consequence of integration by parts and derivative identity \cite[18.19.6]{DLMF}).

For the half-range model, while the coefficients are equal to those given above with $\beta = \alpha$, they take the following form. Let
\begin{displaymath}
f_\EE(x)=\frac12 [f(x)+f(-x)],\qquad f_\OO(t)=\frac12[f(x)-f(-x)],\qquad x\in(0,\infty),
\end{displaymath}
be the even and odd portions of $f$, respectively. Then, using the transformation $x = h(t) = \arctanh\sqrt{\frac{1+t}{2}}$, we have 

\begin{eqnarray}
  \nonumber
  \hat{f}_{2m}&=&\int_{-\infty}^\infty f(x)\varphi_{2m}(x)\D x=2\int_0^\infty f_\EE(x)\varphi_{2m}(x)\D x=2\!\int_{-1}^1 f_\EE(h(t))h'(t)\varphi_{2m}(h(t))\D t\\
  \nonumber
  &=&\frac{2^{-\frac14}}{\sqrt{g_m^{(\alpha,-\frac12)}}} \int_{-1}^1 f_\EE\!\left(\arctanh\sqrt{\frac{1+t}{2}}\right) (1-t)^{-\frac12(1-\alpha)} (1+t)^{-\frac12}\PP_m^{(\alpha,-\frac12)}(t)\D t\\
  \label{eq:4.2}
  &=&\frac{2^{-\frac14}}{\sqrt{g_m^{(\alpha,-\frac12)}}} \int_{-1}^1 \frac{f_\EE\!\left(\arctanh\sqrt{\frac{1+t}{2}} \right) }{(1-t)^{(1+\alpha)/2}} \PP_m^{(\alpha,-\frac12)}(t) w_\EE(t)\D t.
\end{eqnarray}

Likewise,
\begin{eqnarray}
  \nonumber
  \hat{f}_{2m+1}&=&\int_{-\infty}^\infty \varphi_{2m}(x)\D x=2\int_0^\infty f_\OO(x)\varphi_{2m+1}(x)\D x\\
  \nonumber
  &=&2\int_{-1}^1 f_\OO(h(t))h'(t)\varphi_{2m+1}(h(t))\D t\\
  \label{eq:4.3}
  &=&-\frac{2^{\frac14}}{\sqrt{g_m^{(\alpha,\frac12)}}} \int_{-1}^1 \frac{f_\OO\left(\arctanh\sqrt{\frac{1+t}{2}} \right)}{(1-t)^{(1+\alpha)/2}(1+t)^{\frac12}} \PP_m^{(\alpha,\frac12)}(t) w_\OO(t)\D t.
\end{eqnarray}
The convergence of the coefficients is determined by the functions
\begin{displaymath}
F_\EE(t) = \frac{f_\EE\!\left(\arctanh\sqrt{\frac{1+t}{2}}\right) }{(1-t)^{(1+\alpha)/2}}, \qquad  F_\OO(t) = \frac{f_\OO\left(\arctanh\sqrt{\frac{1+t}{2}} \right)}{(1-t)^{(1+\alpha)/2}(1+t)^{\frac12}}.
\end{displaymath}

Exactly how properties of the function $f$ itself determine the rate of convergence is a topic for another paper. It is clear from the outset that it is some combination of regularity and decay at infinity for the function $f$ which will determine the regularity of $F$ --- if for example $f$ only decays algebraically at infinity, then $F$ will be unbounded in $(-1,1)$!

With regards to computation, the first $N$ Jacobi coefficients of a given function can be approximated in $\O{N (\log N)^2}$ operations using fast polynomial transform techniques described in \cite{townsend18fpt}. An efficient and straightforward Julia implementation exists in the software package, \textsc{ApproxFun} \cite{approxfun}.

The cases where $\alpha, \beta = \pm \frac12$ correspond to Chebyshev polynomials of  various kinds. There are four kinds of Chebyshev polynomials:
\begin{enumerate}
  \item Chebyshev polynomials of the first kind $\CC{T}_m$, orthogonal in $(-1,1)$ with the weight function $(1-t^2)^{-\frac12}$;
  \item Chebyshev polynomials of the second kind $\CC{U}_m$, orthogonal in $(-1,1)$ with the weight function $(1-t^2)^{\frac12}$;
  \item Chebyshev polynomials of the third kind $\CC{V}_m$, orthogonal in $(-1,1)$ with the weight function $(1-t)^{\frac12}(1+t)^{-\frac12}$; and
  \item Chebyshev polynomials of he fourth kind $\CC{W}_m$, orthogonal in $(-1,1)$ with the weights function$(1-t)^{-\frac12}(1+t)^{\frac12}$
\end{enumerate}
\cite[18.3]{DLMF}. They are all scaled Jacobi polynomials with $\alpha,\beta\in\{-\frac12,+\frac12\}$ and are expressible in terms of trigonometric functions,
\begin{displaymath}
\begin{array}{lcl}
\displaystyle  \CC{T}_m(\cos\theta)=\cos m\theta, &\qquad& \displaystyle \CC{U}_m(\cos\theta)=\frac{\sin(m+1)\theta}{\sin\theta},\\[6pt]
\displaystyle \CC{V}_m(\cos\theta)=\frac{\sin(m+\frac12)\theta}{\sin\frac12\theta}, && \displaystyle \CC{W}_m(\cos\theta)=\frac{\cos(m+\frac12)\theta}{\cos\frac12\theta}
\end{array}
\end{displaymath}
\cite[18.5.1--4]{DLMF}.  Therefore, letting $t=\cos\theta$, the integral in equation \eqref{eq:4.1} can be converted into trigonometric integrals over $(0,\pi)$ which correspond to either Cosine or Sine Transform. Either can be discretised and computed in $\O{N\log_2N}$ operations using Fast Cosine Transform (FCT) or Fast Sine Transform (FST). 

Because of this algorithmic advantage, and odd-even symmetry, the bases $\Phi^{(\alpha,\beta)}$ with $(\alpha,\beta) = \left(-\frac12,-\frac12\right)$ and $(\alpha,\beta) = \left(\frac12,\frac12\right)$, {are our preferred choice of parameters}. The Tanh-Chebyshev-T functions are given by
\begin{eqnarray*}
  \varphi^{\left(-\frac12,-\frac12\right)}_m(x) &=& (-1)^m\sqrt{2/\pi} \, \mathrm{sech}^{\frac12} x \CC{\tilde{T}}_m(\tanh x), \qquad m\in\BB{Z}_+, \\ 
  b^{\left(-\frac12,-\frac12\right)}_m &=& \frac12\!\left(m+\frac12\right),
\end{eqnarray*}
where $\CC{\tilde{T}}_m = \CC{T}_m$ if $m>0$ and $\CC{\tilde{T}}_0 = 1/\sqrt{2}$. The Tanh-Chebyshev-U functions are given by
\begin{eqnarray*}
  \varphi^{\left(\frac12,\frac12\right)}_m(x) &=& (-1)^m \sqrt{\frac{2}{\pi}} \, \mathrm{sech}^{\frac32} x \CC{U}_m(\tanh x), \qquad m\in\BB{Z}_+, \\ 
  b^{\left(\frac12,\frac12\right)}_m &=& \frac12\!\left(m+\frac32\right).
\end{eqnarray*}

For the half-range model, equations \R{eq:4.2} and \R{eq:4.3} can also be converted into Sine and Cosine transforms. Specifically, with $\alpha=-\frac12$ we obtain the combination of Chebyshev polynomials of the first and the fourth kind, $\CC{T}_m$ and $\CC{W}_m$, which can be computed with FCT-I and FCT-II respectively -- cf.\ \cite{duhamel90fft} for different flavours of FCT and FST. Likewise, $\alpha=\frac12$ results in a combination of $\CC{U}_m$ and $\CC{V}_m$, computable with FST-I and FST-II respectively.

\subsection{The Fourier transform representation and completeness}

We will now prove that these bases are complete in $\CC{L}_2(\BB{R})$. By \cite[Thms 6, 8]{iserles18osssd} (see also equation \eqref{eq:1.5}) there exists a complex-valued function $g_{\alpha,\beta}(\xi)$ with even real part and odd imaginary part such that
\begin{equation}\label{eq:4.4}
\varphi^{(\alpha,\beta)}_m(x) = \frac{(-\ii)^m}{\sqrt{2\pi}} \int_{-\infty}^\infty g_{\alpha,\beta}(\xi) p_m(\xi) \ee^{\ii x \xi} \,\D\xi,
\end{equation}
where $P = \{p_m\}_{m\in\BB{Z}_+}$ are the orthonormal polynomials with respect to $\D\mu(\xi) = w(\xi) \, \D\xi = |g_{\alpha,\beta}(\xi)|^2 \, \D\xi$. By \cite[Thm.~9]{iserles18osssd}, the functions $\Phi^{(\alpha,\beta)}$ are complete in $\CC{L}_2(\BB{R})$ if $w(\xi) > 0$ for all $\xi \in \BB{R}$ and polynomials are dense in the space $\CC{L}_2(\BB{R}, \D\mu(\xi))$.

Now, applying the Fourier transform to both sides of equation \eqref{eq:4.4} and setting $m=0$, we have
\begin{displaymath}
g_{\alpha,\beta}(\xi) \propto \int_{-\infty}^\infty \varphi_0^{(\alpha,\beta)}(x) \ee^{-\ii x \xi} \,\D x.
\end{displaymath}
We can normalise this function after we find an expression. Let us use the change of variables $\tau = \ee^x$ to manipulate this integral into a more reasonable form.
\begin{eqnarray*}
  & & \int_{-\infty}^\infty \varphi_0^{(\alpha,\beta)}(x) \ee^{-\ii x \xi} \,\D x \\
  &=& \int_{-\infty}^\infty (1-\tanh x)^{\frac{\alpha+1}{2}}(1+\tanh x)^{\frac{\beta+1}{2}} \ee^{-\ii x \xi} \,\D x \\
  &=&  \int_{-\infty}^\infty \left(\frac{2\ee^{-x}}{\ee^x + \ee^{-x}}\right)^{\frac{\alpha+1}{2}}\left(\frac{2\ee^{x}}{\ee^x + \ee^{-x}}\right)^{\frac{\beta+1}{2}} \ee^{-\ii x \xi} \,\D x \\
  &=& 2^{\alpha + \beta + 1} \int_0^\infty \frac{\tau^{-\frac{\alpha+1}{2}} \tau^{\frac{\beta+1}{2}} \tau^{-\ii \xi}}{(\tau+\tau^{-1})^{\frac{\alpha+\beta}{2}+1}} \, \frac{\D \tau}{\tau} \\
  &=& 2^{\alpha + \beta+1} \int_0^\infty \frac{\tau^{\beta-\ii \xi}}{(1+\tau^2)^{\frac{\alpha+\beta}{2}+1}} \, \frac{\D \tau}{\tau}.
\end{eqnarray*}
The change of variables $\sigma = \tau^2$ transforms this into
\begin{displaymath}
\int_{-\infty}^\infty \varphi_0^{(\alpha,\beta)}(x) \ee^{-\ii x \xi} \,\D x = 2^{\alpha+\beta} \int_0^\infty \frac{\sigma^{\frac{\beta-1}{2}-\frac{\ii\xi}{2}}}{(1+\sigma)^{\frac{\alpha+\beta}{2}+1}} \, \D \sigma
\end{displaymath}
This is one of the standard integral formulae for Euler's Beta function \cite[5.12.3]{DLMF}, meaning that
\begin{displaymath}
g_{\alpha,\beta}(\xi) \propto \CC{B}\!\left(\frac{\alpha+1}{2} + \frac{\ii\xi}{2}, \frac{\beta+1}{2} - \frac{\ii\xi}{2}\right).
\end{displaymath}
Using the identity $\CC{B}(a,b) = \Gamma(a)\Gamma(b)/\Gamma(a+b)$ for all $a,b\in\BB{C}$, we deduce the existence of a real constant $C_{\alpha,\beta}$ such that
\begin{displaymath}\label{eqn:gab}
g_{\alpha,\beta}(\xi) = C_{\alpha,\beta} \Gamma\left(\frac{\alpha+1}{2} + \frac{\ii\xi}{2} \right)\Gamma\left(\frac{\beta+1}{2} - \frac{\ii\xi}{2} \right).
\end{displaymath}
Note that $g_{\alpha,\beta}$ is complex-valued in general, but has an even real part and odd imaginary part, which implies that $\varphi^{\alpha,\beta}_m(x)$ is real-valued for all $x \in \BB{R}$. Barnes's Beta Integral \cite[5.13.3]{DLMF} can be used to compute the constant $C_{\alpha,\beta}$ which makes the measure $ \D\mu_{\alpha,\beta}(\xi) = |g_{\alpha,\beta}(\xi)|^2 \D \xi$ have mass equal to one. 

We have thus proved the following theorem.
\begin{theorem}
  \label{measure}
  The Fourier-space measure associated with $\{\varphi_m^{(\alpha,\beta)}\}_{m=0}^\infty$ is
  \begin{displaymath}
  \D\mu_{\alpha,\beta}(\xi) = C_{\alpha,\beta}^2 \left|\Gamma\left(\frac{\alpha+1}{2} + \frac{\ii\xi}{2} \right)\Gamma\left(\frac{\beta+1}{2} - \frac{\ii\xi}{2} \right)\right|^2 \!\D\xi,\qquad \alpha,\beta>-1.
  \end{displaymath}
\end{theorem}

There is little hope of simplifying the expression for $\D\mu_{\alpha,\beta}$ from Theorem~\ref{measure} for general $\alpha,\beta>-1$ except in some special cases.

Let us first consider the  case $\beta = -\alpha$, where necessarily $\alpha \in (-1,1)$. In this case, the reflection formula, $\Gamma(z)\Gamma(1-z) = \pi \mathrm{cosec}\,\pi z$ \cite[5.5.3]{DLMF}, implies
\begin{displaymath}
g_{\alpha,-\alpha}(\xi) = \pi C_{\alpha,-\alpha}\frac{1}{\cos\left( \frac{\pi}{2} (\alpha + \ii\xi) \right)},
\end{displaymath}
and using the basic multiple angle formula, the associated measure is equal to
\begin{displaymath}
\D\mu_{\alpha,-\alpha}(\xi) = 2\pi^2 C_{\alpha,-\alpha}^2 \frac{1}{\cosh(\pi\xi) + \cos(\pi\alpha)} \,\D\xi.
\end{displaymath}
The polynomials associated with this measure are known as the \emph{Carlitz polynomials} (after mapping them from a line in the complex plane to the real line) \cite{carlitz59ben}. They were mentioned in \cite{iserles18osssd}, but now we have the full picture of their relationship to orthogonal systems with a tridiagonal, skew-symmetric differentiation matrix.

The case $\alpha = \beta$ yields
\begin{equation}\label{eq:4.5}
g_{\alpha,\alpha}(\xi) = C_{\alpha,\alpha} \left|\Gamma\left(\frac{\alpha+1}{2} + \frac{\ii\xi}{2} \right)\right|^2\!,
\end{equation}
since $\Gamma(\overline{z}) = \overline{\Gamma(z)}$.

\begin{lemma}
  \label{g_integer}
  In the case $\alpha=\beta\geq0$ being integers we have
  \begin{eqnarray}
    \label{eq:4.6}
    g_{2n,2n}(\xi)&=&\pi C_{n,n}\frac{\prod_{j=0}^{n-1}[(j+\frac12)^2+\frac14 \xi^2]}{\cosh (\frac{\pi\xi}{2})},\qquad \\
    \nonumber
    g_{2n+1,2n+1}(\xi)&=&\frac{\pi}{2} C_{2n+1,2n+1} \frac{\xi\prod_{j=1}^n (j^2+\frac14\xi^2)}{\sinh(\frac{\pi\xi}{2})},\qquad n\in\BB{Z}_+.
  \end{eqnarray}
\end{lemma}

\begin{proof}
  In the case $n=0$ \R{eq:4.6} is confirmed by direct computation. Otherwise we use the standard recurrence formula $\Gamma(z+1)=z\Gamma(z)$. \R{eq:4.5} implies that
  \begin{displaymath}
  \frac{g_{n,n}(\xi)}{C_{n,n}}=\frac{(n-1)^2+\xi^2}{4} \frac{g_{n-2,n-2}(\xi)}{C_{n-2,n-2}},\qquad n\geq2,
  \end{displaymath}
  and \R{eq:4.6} follows by simple induction. 
\end{proof}

Regarding completeness of $\Phi^{(\alpha,\beta)}$ in $\CC{L}_2(\BB{R})$, as mentioned above,  $g_{\alpha,\beta}$ being nonzero for all $\xi \in \BB{R}$ (because the $\Gamma$ function has no roots in the complex plane), all that remains is the question of density of polynomials in the space $\CC{L}_2(\BB{R},\D\mu)$. It would be sufficient that $w(\xi)$ have exponential decay as $\xi \to \pm\infty$ \cite[pp.~45,86,86]{akhiezer1965classical}.

To show that these measures have exponential decay, we use the asymptotic formula \cite[5.11.9]{DLMF},
\begin{displaymath}
|\Gamma(x + \ii y)| \sim  \sqrt{2\pi} |y|^{x-\frac12}\ee^{-\pi|y|/2} \text{ as } |y| \to \infty.
\end{displaymath}
This implies,
\begin{displaymath}
\left|\Gamma\left(\frac{\alpha+1}{2} + \frac{\ii\xi}{2} \right)\Gamma\left(\frac{\beta+1}{2} - \frac{\ii\xi}{2} \right)\right|^2 \sim 4\pi^2 \left|\frac{\xi}{2}\right|^{\alpha+\beta} \ee^{-\pi\xi} \text{ as } |\xi| \to \infty,
\end{displaymath}
as required.

An intriguing fact is that \R{eq:4.5} can be alternatively derived from a little-known formula due to Ramanujan \cite{ramanujan15sdi}, namely
\begin{displaymath}
\int_{-\infty}^\infty |\Gamma(a+\ii \xi)|^2 \ee^{\ii x\xi}\D \xi=\frac{\sqrt{\pi}\,\Gamma(a)\Gamma(a+\frac12)}{\cosh^{2a}\left(\frac{x}{2}\right)},\qquad a>0.
\end{displaymath}
Taking $a=(\alpha+1)/2$, a trivial change of variable takes $g_{\alpha,\alpha}$ to the correct $\varphi_0$ and the proof follows by inverting the Fourier transform. Going in the reverse direction, we can obtain a  generalisation of Ramanujan's formula, to the Fourier transform of $\Gamma(a+\ii\xi)\Gamma(b-\ii\xi)$ for $a,b > 0$.

Except for Carlitz polynomials and their immediate generalisations, orthogonal polynomials associated with measures $\D\mu_{\alpha,\beta}$ do not appear to have been studied in the literature. They might be an interesting object for further study, being examples of measures supported by the entire real line, yet distinct from the more familiar Freud-type measures.

\setcounter{equation}{0}
\setcounter{figure}{0}
\section{Conclusions}

We set ourselves a goal in this paper: to identify and characterise orthonormal systems, complete in $\CC{L}_2(\BB{R})$ and with a skew-symmetric, tridiagonal, irreducible differentiation matrix whose expansion coefficients can be computed rapidly. In particular we are interested in the computation of the first $N$ expansion coefficients in $\O{N\log_2N}$ operations, utilising familiar transforms, e.g.\ FFT, FCT or FST. 

In Section~2 we introduced the Tanh-Jacobi functions, of which the four cases where $\alpha,\beta = \pm \frac12$ have expansion coefficients which can be computed using FCTs or FSTs. Subsequently, in Section~3 we described two kinds of half-range expansions (i.e., treating the even and the odd part of a function separately) which can be computed rapidly with either FCT or FST: the first based on a combination of Chebyshev polynomials of the first and the fourth kind, the second on such polynomials of the second and third kind. 

Mathematically, the two approaches are identical when $\alpha = \beta$, although their computation is somewhat different. Which is preferable? As things stand, there is no clear answer (and things are complicated by the availability of yet another approach of this kind, using skew-Hermitian differentiation matrices, which is described in  \cite{iserles19for}). The full-range approximation has the virtue of simplicity, hence of easier implementation. A possible advantage of a half-range approximation is more subtle. Once $\mathcal{D}$ approximates the first derivative, $\mathcal{D}^2$ approximates the second one and, provided $\mathcal{D}$ is skew symmetric, $\mathcal{D}^2$ is negative semi-definite -- this is in line with the Laplace operator being negative semi-definite and is vital for stability. 

As a square of a skew-symmetric, tridiagonal matrix, $\mathcal{D}^2$ neatly separates even and odd functions. Specifically, let $\GG{E}\oplus\GG{O}=\CC{L}_2(\BB{R})\cap\CC{C}^2(\BB{R})$ be a representation of square-integrable, twice differentiable functions on the line as a direct sum of even functions $\GG{E}$ and odd functions $\GG{O}$. A derivative takes $\GG{E}$ to $\GG{O}$ and vice versa, hence a second derivative is invariant in both $\GG{E}$  and  $\GG{O}$. There is thus a virtue, at least once both first and second derivatives are present, to work separately in $\GG{E}$ and $\GG{O}$, as is  the case with half-range approximations. 

Is simplicity preferable to even--odd separation? Are there additional considerations at play? By this stage it is impossible to provide a definitive answer. The purpose of the paper is to present a range of new results that improve our knowledge of approximation on the real line in the context of spectral methods.

The outlook for spectral methods on the real line using Tanh-Chebyshev-T functions appears promising. Consider the basic first order differential operator,
\begin{displaymath}
\mathcal{L}u(x) = u'(x) + a(x)u(x),
\end{displaymath}
where $a$ is a bounded function such that a rapidly convergent expansion of the form $a(x) = \sum_{m=0}^\infty a_m \CC{\tilde{T}}_m(\tanh x)$ is possible\footnote{With the absence of the $\mathrm{sech}^{\frac12}x$ weight, these functions need not be square-integrable on the real line, which is perfectly consistent with which functions $a(x)$ allow $\mathcal{L}$ to be a bounded operator on the Sobolev space $\CC{H}^1(\BB{R})$.}. In coefficient space for the Tanh-Chebyshev-T functions, the operator $\mathcal{L}$ becomes the infinite-dimensional matrix $\mathcal{D} + \mathcal{A}$, where $\mathcal{D}$ is the skew-symmetric tridiagonal differentiation matrix \eqref{eq:1.3} with $b_m = \frac12\left( m + \frac12 \right)$, and
\begin{equation*}
\mathcal{A} =
\left[
\begin{array}{ccccc}
a_0 & a_1 & a_2 & a_3 & \cdots\\
a_1 & a_0 & a_1 & a_2 & \ddots \\
a_2 & a_1 & a_0 & a_1 & \ddots \\
a_3 & a_2 & a_1 & a_0 & \ddots \\
\vdots & \ddots & \ddots & \ddots & \ddots
\end{array}
\right] + \left[
\begin{array}{ccccc}
a_1 & a_2 & a_3 & a_4 & \cdots\\
a_2 & a_3 & a_4 & a_5 & \iddots \\
a_3 & a_4 & a_5 & a_6 & \iddots \\
a_4 & a_5 & a_6 & a_7 & \iddots \\
\vdots & \iddots & \iddots & \iddots & \iddots
\end{array}
\right].
\end{equation*}

The matrix is Toeplitz-plus-Hankel and, if $a$ is sufficiently regular then the matrix is effectively banded, because if $a_m = 0$ for $m > M$ for some integer $M$ then the resulting operator has bandwidth $M$. This implies that an Olver--Townsend type approach for infinite-dimensional QR solution is possible in principle. An $r$th order differential operator with variable coefficients whose expansions have a maximum of $M$ terms yields a matrix with bandwidth of at most $r + M$. Furthermore, the resulting matrices respect certain symmetries that the operator $\mathcal{L}$ may have, something which is not true of the original Olver--Townsend ultraspherical spectral method. For example, if $\mathcal{L}$ is self-adjoint on $\CC{L}_2(\BB{R})$ then $L$ is a self-adjoint operator on $\ell_2$.

One final, credible, but as of yet unexplored, application of this work is for the computation of the Fourier transform of certain functions on the real line. Let $f \in \CC{L}_2(\BB{R})$ have a rapidly convergent expansion in one of the Tanh-Jacobi bases, $f(x) = \sum_{m=0} c_m \varphi_m^{(\alpha,\beta)}(x)$. An approximation $f_N(x) = \sum_{m=0}^N c^N_m \varphi_m^{(\alpha,\beta)}(x)$ computed using either fast polynomial transforms or the FCT/FST as discussed in Section 4. By the identity in equation \eqref{eq:1.5} and \cite{iserles18osssd}, the Fourier transform of $f_N$ is an expansion in the generalised Carlitz polynomials weighted by $g_{\alpha,\beta}$ (see equation \eqref{eqn:gab}),
\begin{equation}
\mathcal{F}[f](\xi) = g_{\alpha,\beta}(\xi) \sum_{m=0}^N (-\ii)^m c^N_m p_m(\xi).
\end{equation}
By Theorem 6 of \cite{iserles18osssd}, these generalised Carlitz polynomials, which are orthonormal with respect to $\D \mu_{\alpha,\beta}(\xi)$ in Theorem 3, satisfy
\begin{equation}
p_{m+1}(\xi) = \frac{\xi}{b_m} p_m(\xi) - \frac{b_{m-1}}{b_m} p_{m-1}(\xi),
\end{equation}
where $b_m$ is given in equation \eqref{eq:2.8}. Clenshaw's algorithm can be used to evaluate this expansion at a single point $\xi \in \BB{R}$, using purely the coefficients $\{b_m\}_{m\in\BB{Z}_+}$ \cite{clenshaw1955note}. This is similar to the work of Weber in which one computes a series expansion with orthogonal rational functions whose Fourier transforms are Laguerre functions, utilising the Fast Fourier Transform \cite{weber1980numerical}.

{A major issue that we have not pursued in this paper is the speed of convergence of different orthonormal bases in $\CC{L}_2(\BB{R})$, not least  Tanh-Jacobi bases. While approximation theory of analytic functions in a finite interval by means of an analytic orthonormal basis is well known, this is not the case on the real line. The issue (which persists under a map from $\BB{R}$ to a finite interval) is that this theory requires both the underlying function and the orthonormal basis to be analytic in a larger ellipse surrounding a finite interval: it is the size and shape of that ellipse that determines the exponential speed of convergence. This becomes problematic on the real line. Some orthonormal bases (e.g.\ Hermite functions) have an essential singularity at infinity (viewed as the North Pole of the Riemann sphere), as does, for example, the tanh map. Even when, like the Malmquist--Takenaka system \cite{iserles19for}, a basis is analytic in a strip surrounding $\BB{R}$, our problems are not over because most analytic functions of interest are likely to have an essential singularity at infinity. As a striking example, while the Malmquist--Takenaka expansion coefficients of $1/(1+x^2)$ (which is analytic ina  strip about $\BB{R}$ decay exponentially, as $3^{-|n|}$, the speed of decay for $\sin x/(1+x^2)$ is $\O{|n|^{-5/4}}$, barely better than linear \cite{weideman1995computing}! While important results have been published by Boyd \cite{boyd87smu}, much more needs be done to understand approximation theory on the real line.}

        

\appendix
\section{}




\ack 


The second author is grateful to FWO Research Foundation Flanders for a postdoctoral research fellowship at KU Leuven which he enjoyed during the research and writing of this paper.


\frenchspacing
\bibliographystyle{cpam}
\bibliography{CPAMstylePaper3}

\begin{thebibliography}{10}
\providecommand{\url}[1]{\texttt{#1}}
\providecommand{\urlprefix}{Available at: }
\providecommand{\eprint}[2][]{\url{#2}}

\bibitem{akhiezer1965classical}
Akhiezer, N.~I. \emph{The classical moment problem: and some related questions
  in analysis}, vol.~5, Oliver \& Boyd, 1965.

\bibitem{bader14eas}
Bader, P.; Iserles, A.; Kropielnicka, K.; Singh, P. Effective approximation for
  the semiclassical {S}chr\"odinger equation. \emph{Found. Comput. Math.}
  \textbf{14} (2014), no.~4, 689--720.
\urlprefix\url{https://doi.org/10.1007/s10208-013-9182-8}

\bibitem{bader16eml}
Bader, P.; Iserles, A.; Kropielnicka, K.; Singh, P. Efficient methods for
  linear {S}chr\"{o}dinger equation in the semiclassical regime with
  time-dependent potential. \emph{Proc. Royal Soc. A.} \textbf{472} (2016), no.
  2193, 20150\,733, 18.
\urlprefix\url{https://doi.org/10.1098/rspa.2015.0733}

\bibitem{blanes17hoc}
Blanes, S.; Casas, F.; Thalhammer, M. High-order commutator-free quasi-{M}agnus
  exponential integrators for non-autonomous linear evolution equations.
  \emph{Comput. Phys. Commun.} \textbf{220} (2017), 243--262.
\urlprefix\url{https://doi.org/10.1016/j.cpc.2017.07.016}

\bibitem{boyd87orf}
Boyd, J.~P. Orthogonal rational functions on a semi-infinite interval. \emph{J.
  Comput. Phys.} \textbf{70} (1987), no.~1, 63--88.
\urlprefix\url{https://doi.org/10.1016/0021-9991(87)90002-7}

\bibitem{boyd87smu}
Boyd, J.~P. Spectral methods using rational basis functions on an infinite
  interval. \emph{J. Comput. Phys.} \textbf{69} (1987), no.~1, 112--142.
\urlprefix\url{https://doi.org/10.1016/0021-9991(87)90158-6}

\bibitem{carlitz59ben}
Carlitz, L. Bernoulli and {E}uler numbers and orthogonal polynomials.
  \emph{Duke Math. J.} \textbf{26} (1959), 1--15.
\urlprefix\url{http://projecteuclid.org/euclid.dmj/1077468334}

\bibitem{chihara78iop}
Chihara, T.~S. \emph{An {I}ntroduction to {O}rthogonal {P}olynomials}, Gordon
  and Breach Science Publishers, New York--London--Paris, 1978. Mathematics and
  its Applications, Vol. 13.

\bibitem{clenshaw1955note}
Clenshaw, C. A note on the summation of {C}hebyshev series. \emph{Math. Comp.}
  \textbf{9} (1955), no.~51, 118--120.

\bibitem{deift99sao}
Deift, P.; Kriecherbauer, T.; McLaughlin, K. T.-R.; Venakides, S.; Zhou, X.
  Strong asymptotics of orthogonal polynomials with respect to exponential
  weights. \emph{Comm. Pure Appl. Math.} \textbf{52} (1999), no.~12,
  1491--1552.
\urlprefix\url{https://doi.org/10.1002/(SICI)1097-0312(199912)52:12<1491::AID-CPA2>3.3.CO;2-R}

\bibitem{duhamel90fft}
Duhamel, P.; Vetterli, M. Fast {F}ourier transforms: a tutorial review and a
  state of the art. \emph{Signal Process.} \textbf{19} (1990), no.~4, 259--299.
\urlprefix\url{https://doi.org/10.1016/0165-1684(90)90158-U}

\bibitem{dutt96fap}
Dutt, A.; Gu, M.; Rokhlin, V. Fast algorithms for polynomial interpolation,
  integration, and differentiation. \emph{SIAM J. Numer. Anal.} \textbf{33}
  (1996), no.~5, 1689--1711.
\urlprefix\url{https://doi.org/10.1137/0733082}

\bibitem{fokas92iam}
Fokas, A.~S.; Its, A.~R.; Kitaev, A.~V. The isomonodromy approach to matrix
  models in {$2$}{D} quantum gravity. \emph{Comm. Math. Phys.} \textbf{147}
  (1992), no.~2, 395--430.
\urlprefix\url{http://projecteuclid.org/euclid.cmp/1104250643}

\bibitem{guo10gjr}
Guo, B.-Y.; Yi, Y.-G. Generalized {J}acobi rational spectral method and its
  applications. \emph{J. Sci. Comput.} \textbf{43} (2010), no.~2, 201--238.
\urlprefix\url{https://doi.org/10.1007/s10915-010-9353-6}

\bibitem{hairer16nsp}
Hairer, E.; Iserles, A. Numerical stability in the presence of variable
  coefficients. \emph{Found. Comput. Math.} \textbf{16} (2016), no.~3,
  751--777.
\urlprefix\url{https://doi.org/10.1007/s10208-015-9263-y}

\bibitem{hairer17bss}
Hairer, E.; Iserles, A. Banded, stable, skew-symmetric differentiation matrices
  of high order. \emph{IMA J. Numer. Anal.} \textbf{37} (2017), no.~3,
  1087--1103.
\urlprefix\url{https://doi.org/10.1093/imanum/drw037}

\bibitem{iserles14ssd}
Iserles, A. On skew-symmetric differentiation matrices. \emph{IMA J. Numer.
  Anal.} \textbf{34} (2014), no.~2, 435--451.
\urlprefix\url{https://doi.org/10.1093/imanum/drt013}

\bibitem{iserles16jps}
Iserles, A. The joy and pain of skew symmetry. \emph{Found. Comput. Math.}
  \textbf{16} (2016), no.~6, 1607--1630.
\urlprefix\url{https://doi.org/10.1007/s10208-016-9321-0}

\bibitem{iserles18mlm}
Iserles, A.; Kropielnicka, K.; Singh, P. Magnus-{L}anczos methods with
  simplified commutators for the {S}chr\"{o}dinger equation with a
  time-dependent potential. \emph{SIAM J. Numer. Anal.} \textbf{56} (2018),
  no.~3, 1547--1569.
\urlprefix\url{https://doi.org/10.1137/17M1149833}

\bibitem{iserles19for}
Iserles, A.; Webb, M. A family of orthogonal rational functions and other
  orthogonal systems with a skew-{H}ermitian differentiation matrix. \emph{J.
  Fourier Anal. Appl.}  (2019). To appear.

\bibitem{iserles18osssd}
Iserles, A.; Webb, M. Orthogonal systems with a skew-symmetric differentiation
  matrix. \emph{Found. Comput. Math.} \textbf{19} (2019), no.~6, 1191--1221.
\urlprefix\url{https://doi.org/10.1007/s10208-019-09435-x}

\bibitem{jin11mcm}
Jin, S.; Markowich, P.; Sparber, C. Mathematical and computational methods for
  semiclassical {S}chr\"{o}dinger equations. \emph{Acta Numer.} \textbf{20}
  (2011), 121--209.
\urlprefix\url{https://doi.org/10.1017/S0962492911000031}

\bibitem{narayan13gwr}
Narayan, A.~C.; Hesthaven, J.~S. A generalization of the {W}iener rational
  basis functions on infinite intervals, {P}art {II}---{N}umerical
  investigation. \emph{J. Comput. Appl. Math.} \textbf{237} (2013), no.~1,
  18--34.
\urlprefix\url{https://doi.org/10.1016/j.cam.2012.06.036}

\bibitem{DLMF}
Olver, F. W.~J.; Lozier, D.~W.; Boisvert, R.~F.; Clark, C.~W., eds.
  \emph{N{IST} {H}andbook of {M}athematical {F}unctions}, U.S. Department of
  Commerce, National Institute of Standards and Technology, Washington, DC;
  Cambridge University Press, Cambridge, 2010. With 1 CD-ROM (Windows,
  Macintosh and UNIX).

\bibitem{approxfun}
Olver, S.: Approxfun. jl v0.10.3. 2018.

\bibitem{olver2013fast}
Olver, S.; Townsend, A. A fast and well-conditioned spectral method. \emph{SIAM
  Rev.} \textbf{55} (2013), no.~3, 462--489.

\bibitem{patriarca1994boundary}
Patriarca, M. Boundary conditions for the {S}chr{\"o}dinger equation in the
  numerical simulation of quantum systems. \emph{Phys. Rev. E} \textbf{50}
  (1994), no.~2, 1616.

\bibitem{rainville60sf}
Rainville, E.~D. \emph{Special {F}unctions}, The Macmillan Co., New York, 1960.

\bibitem{ramanujan15sdi}
Ramanujan, S. Some definite integrals [{M}essenger {M}ath. {\bf 44} (1915),
  10--18]. in \emph{Collected papers of {S}rinivasa {R}amanujan}, pp. 53--58,
  AMS Chelsea Publ., Providence, RI, 2000.

\bibitem{shapiro03qc}
Shapiro, M.; Brumer, P. \emph{Principles of the Quantum Control of Molecular
  Processes}, Wiley-Interscience, Hoboken, N.J., 2003.

\bibitem{townsend16fag}
Townsend, A.; Trogdon, T.; Olver, S. Fast computation of {G}auss quadrature
  nodes and weights on the whole real line. \emph{IMA J. Numer. Anal.}
  \textbf{36} (2016), no.~1, 337--358.
\urlprefix\url{https://doi.org/10.1093/imanum/drv002}

\bibitem{townsend18fpt}
Townsend, A.; Webb, M.; Olver, S. Fast polynomial transforms based on
  {T}oeplitz and {H}ankel matrices. \emph{Math. Comp.} \textbf{87} (2018), no.
  312, 1913--1934.
\urlprefix\url{https://doi.org/10.1090/mcom/3277}

\bibitem{weber1980numerical}
Weber, H. Numerical computation of the {F}ourier transform using {L}aguerre
  functions and the fast {F}ourier transform. \emph{Numer. Math.} \textbf{36}
  (1980), no.~2, 197--209.

\bibitem{weideman1995computing}
Weideman, J. Computing the {H}ilbert transform on the real line. \emph{Math.
  Comp.} \textbf{64} (1995), no. 210, 745--762.

\bibitem{yi12gjr}
Yi, Y.-g.; Guo, B.-y. Generalized {J}acobi rational spectral method on the half
  line. \emph{Adv. Comput. Math.} \textbf{37} (2012), no.~1, 1--37.
\urlprefix\url{https://doi.org/10.1007/s10444-011-9193-4}

\end{thebibliography}

\end{document}